\documentclass[12pt,reqno,a4paper]{amsart}
\usepackage[dvipsnames, table]{xcolor}
\usepackage{amsthm}
\theoremstyle{plain}
\usepackage{amssymb}
\usepackage{marvosym}
\usepackage{bm}
\usepackage{mathrsfs}
\usepackage{enumerate}  
\usepackage{mathtools}
\usepackage{tikz-cd}
\usepackage{tikz}
\usepackage{graphicx}
\usepackage{float}
\usepackage[titletoc,title]{appendix}
\usepackage{marginnote}
\usepackage{todonotes}
\usepackage{etoolbox}
\usepackage[all]{xy}
\makeatletter
\patchcmd{\Ginclude@eps}{"#1"}{#1}{}{}
\makeatother
\usepackage[outdir=./]{epstopdf}
\usepackage[numbers]{natbib}
\usepackage[utf8]{inputenc}
\usepackage[english]{babel}
\usepackage{changepage}
\usepackage{makecell}
\usepackage{enumitem}

\usepackage[colorlinks,citecolor=blue]{hyperref}

\definecolor{lightblue}{HTML}{1F88CD}
\definecolor{lightgrey}{HTML}{727272}
\definecolor{lightblue2}{HTML}{009EC1}
\definecolor{mypink}{HTML}{FD00B0}
\definecolor{lightred}{HTML}{ff4d4d}

\usepackage{hyperref}
\hypersetup{
	colorlinks=true,
    linkcolor={OliveGreen},
    citecolor={lightred},
	urlcolor={black}
}

\newtheorem*{theorem*}{Theorem}
\newtheorem{theorem}{Theorem}[section]

\newtheorem{lemma}[theorem]{Lemma}
\newtheorem{conjecture}[theorem]{Conjecture}

\newtheorem{proposition}[theorem]{Proposition}
\theoremstyle{definition}

\theoremstyle{definition}
\newtheorem{definition}[theorem]{Definition}
\theoremstyle{definition}
\newtheorem{remark}[theorem]{Remark}
\theoremstyle{definition}

\theoremstyle{definition}

\theoremstyle{definition}

\theoremstyle{definition}
\newtheorem{question}[theorem]{Question}
\theoremstyle{definition}

\theoremstyle{definition}
\newtheorem{question!}[theorem]{Question!}
\theoremstyle{definition}

\makeatletter
\newcommand*\sbt{\mathpalette\sbt@{.75}}
\newcommand*\sbt@[2]{\mathbin{\vcenter{\hbox{\scalebox{#2}{$\m@th#1\bullet$}}}}}
\makeatother

\newcommand{\ra}{\rightarrow}
\newcommand{\xra}{\xrightarrow}

\newcommand{\sst}{\subset}

\let\emptyset\varnothing

\newcommand{\bR}{\bm{\mathrm{R}}}
\newcommand{\bL}{\bm{\mathrm{L}}}

\newcommand{\D}{\mathrm{D}}

\newcommand{\CC}{\mathbb{C}}
\newcommand{\PP}{\mathbb{P}}

\newcommand{\ch}{\mathrm{ch}}

\newcommand{\pr}{\mathrm{pr}}

\renewcommand{\Re}{\operatorname{Re}}
\renewcommand{\Im}{\operatorname{Im}}

\DeclareMathOperator{\im}{im}

\DeclareMathOperator{\rk}{rk}

\DeclareMathOperator{\Coh}{\mathrm{Coh}}

\DeclareMathOperator{\Ext}{Ext}

\DeclareMathOperator{\Hom}{Hom}
\DeclareMathOperator{\RHom}{RHom}

\DeclareMathOperator{\ext}{ext}

\DeclareMathOperator{\Pic}{Pic}

\DeclareMathOperator{\cone}{cone}

\DeclareMathOperator{\Gr}{Gr}

\newcommand{\cC}{\mathcal{C}}
\newcommand{\cA}{\mathcal{A}}
\newcommand{\cE}{\mathcal{E}}
\newcommand{\cU}{\mathcal{U}}
\newcommand{\cH}{\mathcal{H}}

\newcommand{\cS}{\mathcal{S}}

\newcommand{\cI}{\mathcal{I}}

\newcommand{\cQ}{\mathcal{Q}}
\newcommand{\Ku}{\mathcal{K}u}

\newcommand{\cD}{\mathcal{D}}

\newcommand{\cN}{\mathcal{N}}
\newcommand{\cJ}{\mathcal{J}}
\newcommand{\cM}{\mathcal{M}}

\DeclareMathOperator{\cF}{\mathcal{F}}
\DeclareMathOperator{\cG}{\mathcal{G}}

\DeclareMathOperator{\oh}{\mathcal{O}}

\usetikzlibrary{decorations.pathmorphing}
\begin{document}

\title[]{Conics on Gushel--Mukai fourfolds, EPW sextics and Bridgeland moduli spaces}


\address{Beijing International Center for Mathematical Research, Peking University, Beijing, China}
\email{hfguo@pku.edu.cn}

\address{Institute for Advanced Study in Mathematics, Zhejiang University, Hangzhou, Zhejiang Province 310030, P. R. China}
\address{College of Mathematics, Sichuan University, Chengdu, Sichuan Province 610064, P. R. China}
\email{zhiyuliu@stu.scu.edu.cn, jasonlzy0617@gmail.com}
\urladdr{sites.google.com/view/zhiyuliu}

\address{Max Planck Institute for Mathematics, Vivatsgasse 7, 53111 Bonn, Germany}
\address{Institut de Mathématiqes de Toulouse, UMR 5219, Université de Toulouse, Université Paul Sabatier, 118 route de
Narbonne, 31062 Toulouse Cedex 9, France}
\email{shizhuozhang@mpim-bonn.mpg.de,shizhuo.zhang@math.univ-toulouse.fr}

\author{Hanfei Guo, Zhiyu Liu, Shizhuo Zhang}
\address{}
\email{}

\begin{abstract}
We identify the double dual EPW sextic $\widetilde{Y}_{A^{\perp}}$ and the double EPW sextic $\widetilde{Y}_A$, associated with a very general Gushel--Mukai fourfold $X$, with the Bridgeland moduli spaces of stable objects of character $\Lambda_1$ and $\Lambda_2$ in the Kuznetsov component $\Ku(X)$. This provides an affirmative answer to a question of Perry--Pertusi--Zhao. As an  application, we prove a conjecture of Kuznetsov--Perry for very general Gushel--Mukai fourfolds. 
\end{abstract}

\maketitle

\setcounter{tocdepth}{1}








\section{Introduction}
\subsection{Hyperkähler varieties as Bridgeland moduli spaces for Kuznetsov components of Fano fourfolds}
Compact hyperkähler varieties are higher-dimensional analogues of K3 surfaces, which are important building blocks of algebraic geometry. However, constructing a compact hyperkähler variety is involved so that only a few of examples are known. The major examples are moduli spaces of stable sheaves on a K3 surface, by the work  \cite{beauville1983varietes}, 
\cite{mukai1984moduli},
\cite{o1995weight}, \cite{yoshioka1999irreducibility}, \cite{yoshioka2001moduli}
and many others.

To produce more examples of hyperkähler varieties, one could consider the moduli spaces of stable objects on a non-commutative K3 surface, which we now briefly explain. In \cite{kuznetsov2004derived}, Kuznetsov constructs a  semi-orthogonal decomposition of the derived category of a cubic fourfold $X$
$$\D^b(X)=\langle\Ku(X),\mathcal{O}_X,\mathcal{O}_X(H),\mathcal{O}_X(2H)\rangle,$$
where $\oh_X(H)$ is the ample line bundle  $\oh_{\mathbb{P}^5}(1)|_X.$ Kuznetsov observes that the non-trivial semi-orthogonal component $\Ku(X)$ is a K3 category in the sense that it has the same Serre functor and Hochschild cohomology as the derived category of a K3 surface. On the other hand, it has now been expected that the Kuznetsov component of a smooth Fano variety encodes essential birational geometric information. In \cite{bayer2017stability}, the authors construct stability conditions on the Kuznetsov component of a series of Fano varieties, including~$\Ku(X)$. In particular, one could construct Bridgeland moduli spaces of stable objects in $\Ku(X)$ with respect to the stability conditions. Under certain circumstances, these moduli spaces provide new examples of hyperkähler varieties.
In the present article, we focus on the case of a Gushel--Mukai(GM) fourfold $X$, which is a degree $10$ and index two Fano variety. Let $V_5$ 
be a complex vector space of dimension five. A general GM fourfold is defined by a smooth transverse intersection of $\mathrm{Gr}(2,V_5)$ with a linear section $\mathbb{P}^8$ and a quadric section $Q$ in $\mathbb{P}^9$ after the \text{Plücker} embedding
$$X:=\mathrm{Gr}(2,V_5)\cap \mathbb{P}^8\cap Q.$$

By the work \cite{debarre2015special}, \cite{debarre2015gushel}, \cite{kuznetsov2018derived}, \cite{debarre2019gushel} and many others, it is shown that GM fourfolds share many similarities with cubic fourfolds. For example, a GM fourfold $X$ also admits a semi-orthogonal decomposition
$$\D^b(X)=\langle\Ku(X),\mathcal{O}_X,\mathcal{U}^\vee,\mathcal{O}_X(H),\mathcal{U}^\vee(H)\rangle,$$
where $\mathcal{U}$ is the pull back of the tautological sub-bundle on $\mathrm{Gr}(2,V_5)$ and $\oh_X(H)$ is the restriction of the \text{Plücker} line bundle $\oh_{\mathbb{P}^9}(1)$. Furthermore,   $\Ku(X)$ is also a K3 category. In particular, there is a rank two lattice inside the numerical Grothendieck group $\mathcal{N}(\Ku(X))$ generated by 
$$\Lambda_1=-2+(H^2-\Sigma')-\frac{1}{2}P,\quad \Lambda_2=-4+2H-\frac{5}{3}L,$$ 
where $\Sigma'$ is the class of a degree $6$ surface. In \cite{perry2019stability}, the authors construct stability conditions on $\Ku(X)$ and they show that for a non-zero primitive Mukai vector $v$ and a generic stability condition $\sigma$, if the moduli space $\mathcal{M}_{\sigma}(v)$ is non-empty, it is a smooth projective hyperkähler variety of dimension $(v,v)+2$. In particular, they prove that if~$X$ is very general, $\mathcal{M}_{\sigma}(\Ku(X),\Lambda_1)$ is either isomorphic to the double dual EPW sextic~$\widetilde{Y}_{A^{\perp}}$ or the double EPW sextic $\widetilde{Y}_A$, where $A$ is the Lagrangian data associated with $X$. Furthermore, as observed in \cite[Section 5.4.1]{perry2019stability}, they expect that there exist two isomorphisms,  $\mathcal{M}_{\sigma}(\Ku(X),\Lambda_1)\cong\widetilde{Y}_{A^{\perp}}$ and $\mathcal{M}_{\sigma}(\Ku(X),\Lambda_2)\cong\widetilde{Y}_A$. 

\subsection{Main Results}
The first main result of our article answers the question of Perry--Pertusi--Zhao.

\begin{theorem}
\label{main_theorem_1}
Let $X$ be a very general GM fourfold. Then, for a generic stability condition $\sigma$ on $\Ku(X)$, we have
\begin{enumerate}
    \item $\mathcal{M}_{\sigma}(\Ku(X),\Lambda_1)\cong\widetilde{Y}_{A^{\perp}}$. 
    \item $\mathcal{M}_{\sigma}(\Ku(X),\Lambda_2)\cong\widetilde{Y}_A$. 
    \item There is an involutive auto-equivalence on $\Ku(X)$ and the induced involution on $\mathcal{M}_{\sigma}(\Ku(X),\Lambda_1)$ coincides with the natural involution on $\widetilde{Y}_{A^{\perp}}$.
\end{enumerate}
\end{theorem}

\begin{remark} \label{general_rmk}
Here \emph{very general} means \emph{general non-Hodge-special}, see Definition \ref{hodge_special}. We only use this assumption in Theorem~\ref{pr_stable_thm} to prove the stability of projection objects. Once Theorem~\ref{pr_stable_thm} is known for general GM fourfold and generic stability conditions, Theorem~\ref{main_theorem_1} can be generalized to the case of general GM fourfolds. 
\end{remark}

We approach Theorem~\ref{main_theorem_1} by projecting objects related to conics to the Kuznetsov component. The Hilbert scheme of conics $F_g(X)$ of a general GM fourfold $X$ is a smooth projective variety of dimension 5. By \cite{iliev2011fano}, there exists a morphism 
\[f: F_{g}(X)\longrightarrow\widetilde{Y}_{A^\perp}\]
and this morphism is an essential $\mathbb{P}^1$-fibration in that $f$ contracts a $\mathbb{P}^1$-family of generic conics, while taking two special types of conics to two different points. Starting with a twisted structure sheaf $\oh_C(H)$ of a conic $C\subset X$, we show that the projection functor to~$\Ku(X)$ produces an essential $\mathbb{P}^1$-fibration over $\mathcal{M}_{\sigma}(\Ku(X),\Lambda_1)$. More precisely, if $X$ is a very general GM fourfold, we prove that the morphism $p\colon F_g(X)\rightarrow\mathcal{M}_{\sigma}(\Ku(X),\Lambda_1)$ induced by the projection functor coincides with the morphism $f$ constructed in \cite{iliev2011fano}. As a result, we prove that $\mathcal{M}_{\sigma}(\Ku(X),\Lambda_1)\cong\widetilde{Y}_{A^{\perp}}$. 

In \cite{kuznetsov2019categorical}, the authors study GM varieties of arbitrary dimension and propose the following conjecture.
\begin{conjecture}[{\cite[Conjecture 1.7]{kuznetsov2019categorical}}] \label{conjecture_KP2019}
If $X$ and $X'$ are GM varieties of the same dimension such that there exists an equivalence $\Ku(X)\simeq\Ku(X')$, then $X$ and $X'$ are birationally equivalent. 
\end{conjecture}

In \cite{JLLZ2021gushelmukai}, we show the conjecture is true for general GM threefolds. In this article, we prove the following theorem.

\begin{theorem}
\label{main_theorem_2}
Let $X$ and $X'$ be very general GM fourfolds. If there is an equivalence $\Ku(X)\simeq\Ku(X')$, then $X$ and $X'$ are period partners or period duals. In particular, $X$ is birational to $X'$. 
\end{theorem}

We will review the definitions of period partners and period duals in Section \ref{geometry_GM_EPW}. Roughly speaking, $X$ and $X'$ have the same period point up to an involution of the period domain.

The idea is very similar to the proof for GM threefolds as in \cite[Theorem~10.1]{JLLZ2021gushelmukai}. The equivalence $\Phi:\Ku(X)\simeq\Ku(X')$ would induce an isomorphism from the moduli space $\mathcal{M}_{\sigma}(\Ku(X),\Lambda_2)$ to either $\mathcal{M}_{\sigma}(\Ku(X'),\Lambda'_2)$ or $\mathcal{M}_{\sigma}(\Ku(X'),\Lambda'_1)$. By Theorem~\ref{main_theorem_1}, the former case shows that $X$ and $X'$ are period partners while the latter shows that they are period dual. In both cases, $X$ is birationally equivalent to $X'$ by \cite[Corollary 4.16, Theorem 4.20]{debarre2015gushel} and \cite[Remark 5.28]{debarre2019gushel}.

\subsection{Related work}
\subsubsection{Hyperkähler varieties as Bridgeland moduli spaces for Kuznetsov components}
In \cite{li2018twisted}, the authors reconstruct the Fano variety of lines for any cubic fourfold and the LLSvS eightfolds for cubic fourfolds not containing a plane as the moduli spaces of stable objects on the Kuznetsov component with primitive Mukai vector $\lambda_1+\lambda_2$ and $2\lambda_1+\lambda_2$ respectively. In \cite{li2020elliptic}, the authors show that a symplectic resolution of the moduli space $\mathcal{M}_{\sigma}(2\lambda_1+2\lambda_2)$ is a hyperkähler variety, deformation equivalent to O'Grady $10$. 

\subsubsection{Birational categorical Torelli for GM varieties}
In \cite{JLLZ2021gushelmukai}, we show that the Kuznetsov component determines the birational isomorphic class for general GM threefolds while in the present article, we prove a similar statement for very general GM fourfolds. Since GM fivefolds and sixfolds are all rational (cf.~\cite[Prop 4.2]{debarre2015gushel}), the analogous statements automatically hold in these cases.

\subsection{Further questions}
\subsubsection{Structure of the moduli space $\mathcal{M}_{\sigma}(\Ku(X),\Lambda_2)$}
It would be interesting to know if the Bridgeland moduli space $\mathcal{M}_{\sigma}(\Ku(X),\Lambda_2)$ can be realized as a Gieseker moduli space on $X$. Inspired by our previous work \cite{JLLZ2021gushelmukai} for GM threefolds, we wonder if the moduli space $\mathcal{M}_{\sigma}(\Ku(X),\Lambda_2)$ is isomorphic to the moduli space $M_X(4-2H+\frac{1}{6}H^3)$ of semistable sheaves on $X$.

\subsubsection{Refined categorical Torelli for GM fourfolds}
The duality conjecture \cite[Theorem 1.6]{kuznetsov2019categorical} tells us that the Kuznetsov component of GM varieties cannot determine the isomorphism class. In \cite{JLLZ2021gushelmukai}, we prove what we called \emph{Refined categorical Torelli theorem} for GM threefolds, meaning that an extra assumption can be made on the equivalence $\Phi:\Ku(X)\simeq\Ku(X')$ of GM threefolds $X$ and $X'$ to deduce that $X\cong X'$. It is natural to see if similar statements can be proved for GM fourfolds, as asked in \cite[Question 6.5]{PS2022categorical}.

\subsection{Organization of the paper}
In Section~\ref{geometry_GM_EPW}, we review the basic terminologies of GM fourfolds and the associated hyperkähler varieties, double EPW sextics and double dual EPW sextics. In Section~\ref{semi-orthogonal_GM}, we introduce the semi-orthogonal decomposition of GM fourfolds and construct an involutive auto-equivalence on the Kuznetsov component. In Section~\ref{stability_condition_GM}, we briefly review the concepts of weak stability conditions on a general triangulated category and stability conditions on the Kuznetsov component $\Ku(X)$ for a GM fourfold $X$. Then we prove some properties of stability conditions on $\Ku(X)$ which will be used later.
In Section~\ref{classfication_conics_GM}, we introduce three types of conics on GM fourfolds. In Section~\ref{projection_objects}, we compute the projection objects for conics of each type. In Section \ref{stability_projection_conics}, we prove the stability of the projection objects of conics. In Section~\ref{moduli_spaces_construction}, we show that the morphism induced by the projection functor coincides with the classical map defined by Iliev--Manivel in \cite{iliev2011fano}, and as a consequence we prove Theorem~\ref{main_theorem_1}. In Section~\ref{KP_conjecture_GM}, using results in Section~\ref{moduli_spaces_construction}, we prove Theorem~\ref{main_theorem_2}.

\subsection{Notation and conventions}

\begin{itemize}
\item We work over $k=\mathbb{C}$.
\item The term K3 surface means a smooth projective K3 surface.
\item We denote the bounded derived category of coherent sheaves on a smooth variety~$X$ by $\D^b(X)$. The derived dual functor $R\mathcal{H}om_X(-,\oh_X )$ on $\D^b(X)$ is denoted by~$\mathbb{D}(-)$.
\item If $X\to Y$ is a morphism between varieties and $F\in \D^b(Y)$, then we often write~$F_X$ for the
pullback of $F$ to $X$. By abuse of notation, if $D$ is a divisor on $Y$, we often still denote its pullback by $D$.
\item We will use $V_i$ to denote a complex vector space of dimension $i$.
\item We use $\hom$ and $\ext^{i}$ to represent the dimension of the vector spaces $\Hom$ and~$\Ext^{i}$.
\item We denote the Hilbert scheme of conics on a variety $X$ by $F_g(X)$, following the notation in \cite{iliev2011fano}.
\item The symbol $\simeq$ denotes an equivalence of categories and a birational equivalence of varieties. The symbol $\cong$ denotes an isomorphism between varieties, complexes or functors.
\end{itemize}

\subsection*{Acknowledgements}
Firstly, it is our pleasure to thank Arend Bayer and Qizheng Yin for very useful discussions on the topics of this project. We would like to thank Sasha Kuznetsov, Kieran G. O’Grady, Alexander Perry, Laura Pertusi and Xiaolei Zhao for helpful comments. The first author would like to thank Guolei Zhong for his suggestions and companion. The third author thanks Tingyu Sun for constant support and encouragement. We also would like to thank the referee for the careful reading and for providing detailed comments. The third author is supported by the ERC Consolidator Grant WallCrossAG, no. 819864.


\section{Geometry of Gushel--Mukai fourfolds and the associated EPW sextics}\label{geometry_GM_EPW}

Let $X$ be an ordinary GM fourfold, which is defined by a smooth transverse intersection of $\mathrm{Gr}(2,V_5)$ with a linear section $\mathbb{P}^8$ and a quadric section $Q$ in $\mathbb{P}^9$ after the \text{Plücker} embedding
$$X:=\mathrm{Gr}(2,V_5)\cap\mathbb{P}^8\cap Q.$$
There is a natural embedding $\gamma_X: X\rightarrow\mathrm{Gr}(2,V_5)$, which is usually called the Gushel map. We define $\oh_X(H):=\oh_{\mathbb{P}^9}(1)|_X$ and $\mathcal{U}:=\gamma_X^{*}\mathcal{U}_{\mathrm{Gr}(2,V_5)}$,
where $\mathcal{U}_{\mathrm{Gr}(2,V_5)}$ is the tautological rank two sub-bundle of $\mathrm{Gr}(2,V_5)$.

We denote $\sigma_{i, j}\in H^{2(i+j)}(\mathrm{Gr}(2,V_5),\mathbb{Z})$ the Schubert cycles of $\mathrm{Gr}(2,V_5)$ for any integers $3\geq i\geq j\geq 0$. By \cite[Proposition 3.4]{debarre2019gushel}, the cohomology group $H^4(X,\mathbb{Z})$ is torsion free.
The image of the embedding 
$\gamma_X^*:H^4(\mathrm{Gr}(2,V_5),\mathbb{Z})\rightarrow H^{4}(X,\mathbb{Z})$ is a rank two sub-lattice generated by 
$\gamma_X^*(\sigma_1)^2$ and $\gamma_X^*(\sigma_2)$.

\begin{definition} \label{hodge_special}
An ordinary GM fourfold $X$ is called \emph{non-Hodge-special} if $$H^{2,2}(X)\cap H^4(X,\mathbb{Z})=\gamma_X^*H^4(\mathrm{Gr}(2,V_5),\mathbb{Z}).$$
It means $H^{2,2}(X)\cap H^4(X,\mathbb{Z})$ is a rank two integral lattice. $X$ is called \emph{Hodge-special} if the lattice $H^{2,2}(X)\cap H^4(X,\mathbb{Z})$ is of rank at least three.
\end{definition}

According to \cite{debarre2015special}, \cite{debarre2019gushel} and  \cite[Section 4.5]{debarre2020gushel}, there is a period map from the moduli stack of GM fourfolds to the period domain
$$\wp_{4}: \mathbf{M}_{4}^{\mathrm{GM}}\longrightarrow\mathscr{D}.$$
In particular, the locus of periods of the Hodge-special GM fourfolds constitute a countably union of hypersurfaces.


In the current paper, we will always assume $X$ to be \emph{very general} in the sense that it is general among the locus of non-Hodge-special GM fourfolds. For a comment on the general case, see Remark \ref{general_rmk}.

\subsection{EPW sextics and conics on Gushel--Mukai fourfolds}\label{epw_O'grady}
As a cubic fourfold admits an associated hyperkähler variety, which is called Fano variety of lines, a general GM fourfold also admits its associated hyperkähler variety. 

Here we briefly review the definition and some properties of EPW sextics introduced by Eisenbud, Popescu, and Walter, in particular their relations with GM varieties.

Let $X$ be an ordinary GM fourfold, following \cite{debarre2019gushel}, one can naturally associate a triple $(A(X),V_5(X),V_6(X))$ with $X$, which is called \emph{a Lagrangian data} of $X$. Here $V_6(X)$ is a six-dimensional vector space, $V_5(X)$ is a hyperplane in $V_6(X)$ and $A(X)\subset\bigwedge^{3}V_{6}(X)$ is Lagarangian with respect to the natural symplectic structure on $\bigwedge^{3}V_{6}(X)$ given by wedge product. When $X$ is clear, we will use the notation $(A,V_5,V_6)$.

Starting from a Lagrangian data, one can construct various varieties which are important to the geometry of GM fourfolds.
For any integer $l\geq 0$, we define closed subschemes
$$Y_{A}^{\geq l}:=\{[U_{1}]\in\mathbf{P}(V_{6})|\mathrm{dim}(A \cap(U_{1}\wedge\bigwedge^{2}V_{6}))\geq l\}\subset \mathbf{P}(V_{6}),$$
$$Y_{A^\perp}^{\geq l}:=\{[U_{5}]\in\mathbf{P}({V_{6}}^\vee)|\mathrm{dim}(A \cap\bigwedge^{3}U_{5})\geq l\}\subset\mathbf{P}({V_{6}}^\vee).$$
At the same time, we set $Y_{A}^{\ell}:=Y_{A}^{\geq\ell}\backslash Y_{A}^{\geq \ell+1}$ and $Y_{A^{\perp}}^{\ell}:=Y_{A^{\perp}}^{\geq \ell} \backslash Y_{A^{\perp}}^{\geq\ell+1}$.

If $X$ is general, we can assume that $A$ is also general. Then $Y_{A}:=Y_{A}^{\geq 1}\subset\mathbf{P}(V_{6})$
is a normal integral sextic hypersurface, which is called an EPW sextic. The fourfold~$Y_{A}$ is singular at the integral surface $Y_{A}^{\geq 2}$. In \cite[Section 1.2]{o2010double}, the author constructs a canonical double cover
$$\widetilde{Y}_A\rightarrow Y_A,$$
branched over the integral surface $Y_A^{\geq2}$, which is called the double EPW sextic. 
Since $A$ is general, $Y_A^{\geq3}=\emptyset$, $\widetilde{Y}_A$ is a smooth hyperkähler fourfold. The analogue statements also hold for $Y_{A^\perp}^{\geq l}$.

Many properties of a GM fourfold $X$ depend on $A(X)$, its even part of the corresponding Lagrangian subspace. Here we introduce two important notions called period partner and period dual.

\begin{definition}[{\cite[Definition 3.20]{debarre2015gushel}}]
Two GM fourfolds $X_1$ and $X_2$ are called \emph{period partners} if there exists an isomorphism $\phi: V_6(X_1)\cong V_6(X_2)$ such that we have $(\bigwedge^{3} \phi)(A(X_{1}))=A(X_{2})$. They are called \emph{period duals} if there exists an isomorphism $\phi: V_6(X_1)\cong V_6(X_2)^{\vee}$ such that $(\bigwedge^{3} \phi)(A(X_{1}))=A(X_{2})^{\perp}$.
\end{definition}

By definition, period partners are constructed by the same Lagrangian subspace $A$, but possibly with different hyperplanes of $V_6$. By  \cite[Theorem 5.1]{debarre2019gushel}, there is an identification between the primitive Hodge structure of a GM fourfold and the one of its associated double EPW sextic. As a corollary,  period partners share the same period point.


According to \cite[Theorem 1.1]{o2006dual}, for any $A\in\mathbb{L}\mathbb{G}(\wedge^3V_6)^{00}$( i.e.~the associated double cover of the EPW sextic and its dual $\widetilde{Y}_A$ and $\widetilde{Y}_{A^\perp}$ are both smooth), the period points of $\widetilde{Y}_A$ and $\widetilde{Y}_{A^\perp}$ differ by an involution $\Bar{r}$. In particular, $\Bar{r}$ is non-trivial, which implies that $\widetilde{Y}_A$ is not isomorphic to $\widetilde{Y}_{A^\perp}$ for general $A$.



In \cite{iliev2011fano}, the authors show that the double EPW sextic can be constructed from the Hilbert scheme of conics on a general GM fourfold. Here we give a short review of their construction. 

Denote by $F_g(X)$ the Hilbert scheme of conics lying on $X$. When $X$ is general, it is a smooth projective variety of dimension five. $F_g(X)$ admits a natural map to a sextic hypersurface $Y_{X}^\vee\subset \mathbb{P}^5$, over which $F_g(X)$ is essentially a fibration in projective lines. By Stein factorization, we get $$F_g(X)\xrightarrow{f} \widetilde{Y}_{X}^\vee\rightarrow Y_{X}^\vee.$$
It turns out $\widetilde{Y}_{X}^\vee$ is also a smooth fourfold and the morphism $f$ is birational to a $\mathbb{P}^1$-bundle. 

Recall that $Y_X\subset\mathbb{P}(I_X(2))\cong\mathbb{P}^5$. The quadrics containing $\Gr(2,V_5)$,  which are sections of $I_{\Gr(2,V_5)}(2)$, are called Pfaffian quadrics. The hyperplane of Pfaffian quadrics in~$\mathbb{P}(I_X(2))\cong\mathbb{P}^5$ is denoted by $H_p$. Then in the dual projective space, $H_p$ defines a point~$h_p$, which is called the \text{Plücker} point.

By \cite[Theorem 3.2]{iliev2011fano}, there are three types of conics on $X$ which are $\tau$-conics, $\rho$-conics and $\sigma$-conics. The locus of the last two types $F^{\rho}_g(X)$ and $F^{\sigma}_g(X)$ are isomorphic to a three-dimensional quadric $Q^3$ and the blow-up of $\mathbb{P}^4$ at a point, respectively. 

\begin{proposition}[{\cite{iliev2011fano}}] \label{IM11contraction}
The morphism $f: F_g(X)\to\widetilde{Y}_{X}^\vee$ is a birational $\mathbb{P}^1$-bundle in that
\begin{enumerate}
    \item $f(F^{\rho}_g(X))=p_1$ and $f(F^{\sigma}_g(X))=p_2$, where $p_1$, $p_2\in\widetilde{Y}_{X}^\vee$ are preimages of the \text{Plücker} point under the double cover $\widetilde{Y}_{X}^\vee\rightarrow Y_{X}^\vee$.
    
    \item $f(F^{\tau}_g(X))=\widetilde{Y}_{X}^\vee-\{p_1,p_2\}$ and 
    the restriction $f|_{F^{\tau}_g(X)}: F^{\tau}_g(X)\to\widetilde{Y}_{X}^\vee-\{p_1,p_2\}$ is a $\mathbb{P}^1$-bundle.
\end{enumerate}
\end{proposition}

Thus the natural holomorphic two-form on $F_g(X)$, which is induced by the generator of~$H^{3,1}(X)$, descends to $\widetilde{Y}_{X}^\vee$. This makes $\widetilde{Y}_{X}^\vee$ a hyperkähler fourfold. Indeed, 
$\widetilde{Y}_{X}^\vee\rightarrow Y_{X}^\vee$ is a double cover and 
the natural involution is anti-symplectic as in \cite[Proposition 4.17]{iliev2011fano}. This implies $Y_{X}^\vee$ is an EPW sextic and $\widetilde{Y}_{X}^\vee$ coincides with the double cover constructed by O'Grady in Section~\ref{epw_O'grady}. In the followings, we will use the notations $Y_{A}$ and $\widetilde{Y}_{A}$ uniformly to refer to an EPW sextic and its double cover.

\section{Semi-orthogonal decomposition and Kuznetsov components of Gushel--Mukai fourfolds}\label{semi-orthogonal_GM}
Let $X$ be a GM fourfold, the derived category $\D^b(X)$ admits a semi-orthogonal decomposition, given by \cite[Prop 2.3]{kuznetsov2018derived}
$$\D^b(X)=\langle\Ku(X), \oh_X,\mathcal{U}^{\vee},\oh_X(H),\mathcal{U}^{\vee}(H)\rangle.$$
By \cite[Prop 2.6]{kuznetsov2018derived}, the Serre functor of $\Ku(X)$ is $S_{\Ku(X)}\cong [2]$. 
In this case, we define the projection functor as 
$\pr_1\coloneqq \bL_{\oh_X}\bL_{\cU^{\vee}}\bL_{\oh_X(H)}\bL_{\cU^{\vee}(H)}$.

Since $\cU^{\vee}$ has rank two and $c_1(\cU^{\vee})=H$, we have an isomorphism $\cU^{\vee}\cong\cU(H)$. Then using Serre duality, there is an alternative semi-orthogonal decomposition 
$$\D^b(X)=\langle \oh_X(-H), \cU, \Ku(X), \oh_X,\cU^{\vee} \rangle.$$
Under this decomposition, we denote the projection functor by $\pr_2:=\mathbf{R}_{\mathcal{U}} \mathbf{R}_{\mathcal{O}_{X}(-H)}\bL_{\mathcal{O}_{X}}\bL_{\mathcal{U}^\vee}$.

We denote the Grothendieck group of $\Ku(X)$ by $K_0(\Ku(X))$ and $\chi$ is the Euler form over it. Its numerical Grothendieck group is defined as $\mathcal{N}(\Ku(X)):=K_0(\Ku(X))/\ker(\chi)$.

\begin{lemma}[{\cite[Proposition 2.25]{kuznetsov2018derived}}]
Let $X$ be a very general $GM$ fourfold, then  $\mathcal{N}(\Ku(X))\cong\mathbb{Z}^2$. Furthermore, it is generated by $\Lambda_1$ and $\Lambda_2$, where $\Lambda_1=-2+(H^2-\Sigma')-\frac{1}{2}P$ and $\Lambda_2=-4+2H-\frac{5}{3}L$. The Euler form $\chi(-,-)$ on $\langle\Lambda_1,\Lambda_2\rangle$ is in the form
\begin{equation}
\left[               
\begin{array}{cc}   
-2 & 0 \\  
0 & -2\\
\end{array}
\right] .
\end{equation}
\end{lemma}
Here $H:=\gamma^*_X\sigma_1$ and  $\Sigma':=\gamma_X^*\sigma_2$. We have
\[\mathrm{ch}(\cU)=2-H+(-\frac{1}{2}H^2+\Sigma')+\frac{1}{3}L-\frac{1}{12}P.\]
By standard computation, we see 
$$H^2\Sigma'=H^2.\gamma_{X}^{*} \sigma_{2}=(\gamma_{X}^{*} \sigma_{1})^2.\gamma_{X}^{*} \sigma_{2}=\sigma_{1}^2.\sigma_{2}=(\sigma_2)^2+\sigma_{1,1}.\sigma_2=6.$$ 
In \cite{pertusi2019double}, the Todd class of $X$ is calculated \[\operatorname{td}(X)=1+H+\left(\frac{2}{3} H^{2}-\frac{1}{12} \Sigma'\right)+\frac{17}{60} H^{3}+\frac{1}{10} H^{4}.\]
Then for any $\kappa=a+bH+(cH^2+d\Sigma')+eL+fP\in\mathcal{N}(\Ku(X))$,
the Euler characteristic is given by  \[\chi(X,\kappa)=(\operatorname{ch}(\kappa).\operatorname{td}(X))_4=a+\frac{17}{6}b+\frac{37}{6}c+\frac{11}{3}d+e+f.\]



    
    




Now we are going to introduce a functor $T$ on $\Ku(X)$, which is defined by $T:=\bL_{\oh_X}\circ \mathbb{D}$.

\begin{proposition}
\label{involution_functor_original}
The functor $T$ is an involutive auto-equivalence on $\Ku(X)$. 
\end{proposition}

\begin{proof}

First we prove that for any object $E\in \Ku(X)$, we have $T(E)\in \Ku(X)$. Indeed, by the definition of $\bL_{\oh_X}$, we have a triangle
\begin{equation} \label{invo}
    \RHom(\oh_X, \mathbb{D}(E))\otimes \oh_X\to \mathbb{D}(E)\to \bL_{\oh_X}(\mathbb{D}(E))=T(E).
\end{equation}

By Serre duality, we have 
\[\Hom(\cU^{\vee}(H), \mathbb{D}(E)[k])=\Hom(\mathbb{D}(E), \cU[4-k])=\Hom(\cU^{\vee}, E[4-k])=0\] 
for any $k$. Similarly, we see $\RHom(\oh_X(H),\mathbb{D}(E))=\RHom(\cU^{\vee},\mathbb{D}(E))=0$, which implies $\mathbb{D}(E)\in \langle \Ku(X), \oh_X \rangle$. As $\oh_X \in \langle \Ku(X), \oh_X \rangle$, we have  $\bL_{\oh_X}(\mathbb{D}(E))\in \langle \Ku(X), \oh_X \rangle$. Thus if we apply $\Hom(\oh_X, -)$ to the triangle (\ref{invo}), we have $\RHom(\oh_X, \bL_{\oh_X}(\mathbb{D}(E)))=0$, which implies $T(E)=\bL_{\oh_X}(\mathbb{D}(E))\in \Ku(X)$.

Then for any object $F\in \Ku(X)$, if we apply $\Hom(\mathbb{D}(F),-)$ to the triangle (\ref{invo}), since $\RHom(\mathbb{D}(F), \oh_X)=\RHom(\oh_X, F)=0$, we obtain natural isomorphisms
\[\RHom(\mathbb{D}(F), \mathbb{D}(E))\cong \RHom(\mathbb{D}(F), T(E))\cong \RHom(T(F), T(E)),\]
thus $T$ is fully faithful. Note that the last isomorphism follows from $T(E)\in \Ku(X)$ and the adjunction of $\bL_{\oh_X}$.

Now applying the functor $\mathbb{D}$ to the triangle (\ref{invo}), we obtain a triangle
\begin{equation}
    \mathbb{D}(T(E))\to E\to \RHom(\oh_X, \mathbb{D}(E))^{\vee}\otimes \oh_X.
\end{equation}
Taking the functor $\bL_{\oh_X}$ again, by $E\in \Ku(X)$ and exceptionality of $\oh_X$, we have $T(T(E))\cong \bL_{\oh_X}E=E$, which shows that $T\circ T\cong \mathrm{id}_{\Ku(X)}$, i.e., $T$ is an involution on $\Ku(X)$.
\end{proof}

\begin{remark}
It is easy to see that the induced action of $T$ on $\mathcal{N}(\Ku(X))$ will fix $\Lambda_1$ and map $\Lambda_2$ to $-\Lambda_2$. 
\end{remark}

\section{Stability conditions on Kuznetsov components of Gushel--Mukai fourfolds}\label{stability_condition_GM}
In this section, we recall Bridgeland stability conditions on a triangulated category and the notion of stability conditions on the Kuznetsov component of an ordinary  GM fourfold. We follow from \cite[Section 2]{bayer2017stability}.

\subsection{Stability conditions}

Let $\cD$ be a triangulated category and $K_0(\cD)$ be its Grothendieck group. Fix a surjective morphism to a finite rank lattice $v : K_0(\cD) \ra \Lambda$. 

\begin{definition}
The \emph{heart of a bounded t-structure} on $\cD$ is an abelian subcategory $\cA \sst \cD$ such that the following conditions are satisfied
\begin{enumerate}
    \item for any $E, F \in \cA$ and $n <0$, we have $\Hom(E, F[n])=0$,
    \item for any object $E \in \cD$, there exists a sequence of morphisms 
    \[ 0=E_0 \xrightarrow{\phi_1} E_1 \xrightarrow{\phi_2} \cdots \xra{\phi_m} E_m=E \]
    such that $\cone(\phi_i)$ is in the form $A_i[k_i]$, for some sequence $k_1 > k_2 > \cdots > k_m$ of integers and $A_i \in \cA$.
\end{enumerate}
\end{definition}

\begin{definition}
Let $\cA$ be an abelian category and $Z : K_0(\cA) \ra \mathbb{C}$ be a group homomorphism such that for any $E \in \cA$ we have $\Im Z(E) \geq 0$ and if $\Im Z(E) = 0$ then $\Re Z(E) < 0$. Then we call $Z$ a \emph{stability function} on $\cA$.
\end{definition}

\begin{definition}
A \emph{stability condition} on $\cD$ is a pair $\sigma = (\cA, Z)$, where $\cA$ is the heart of a bounded t-structure on $\cD$ and $Z : \Lambda \ra \CC$ is a group homomorphism such that 
\begin{enumerate}
    \item the composition $Z \circ v : K_0(\cA) \cong K_0(\cD) \ra \CC$ is a stability function on $\cA$. From now on, we write $Z(E)$ rather than $Z(v(E))$.
\end{enumerate}
Much like the slope from classical $\mu$-stability, we can define a \emph{slope} $\mu_\sigma$ for $\sigma$ using $Z$. For any $E \in \cA$, set
\[
\mu_\sigma(E) := \begin{cases}  - \frac{\Re Z(E)}{\Im Z(E)}, & \Im Z(E) > 0 \\
+ \infty , & \text{else}.
\end{cases}
\]
We say an object $0 \neq E \in \cA$ is $\sigma$-(semi)stable if $\mu_\sigma(F) < \mu_\sigma(E)$ (respectively $\mu_\sigma(F) \leq \mu_\sigma(E)$) for any proper subobject $F \sst E$. 
\begin{enumerate}[resume]
    \item Any object $E \in \cA$ has a Harder--Narasimhan filtration in terms of $\sigma$-semistability defined above.
    \item There exists a quadratic form $Q$ on $\Lambda \otimes \mathbb{R}$ such that $Q|_{\ker Z}$ is negative definite  and $Q(E) \geq 0$ for all $\sigma$-semistable objects $E \in \cA$. This is known as the \emph{support property}.
\end{enumerate}
\end{definition}

\subsection{Stability conditions on Kuznetsov components of Gushel--Mukai fourfolds}
In \cite{perry2019stability}, inspired by the idea of constructing stability conditions on Kuznetsov components of cubic fourfolds in \cite{bayer2017stability}, the authors embed the Kuznetsov component~$\Ku(X)$ of a general ordinary GM fourfold $X$ into a twisted derived category of modules over a quadric threefold, associated to a conic fibration of $X$. In particular, they prove the following theorem. 

\begin{theorem}
\label{existence_stability_GM4}
Let $X$ be a GM fourfold. Then the category $\Ku(X)$ has a stability condition. 
\end{theorem}

We prove some simple properties for stability conditions on $\Ku(X)$, whose Serre functor is $S_{\Ku(X)}\cong[2]$.

\begin{proposition}
\label{homological_dimensional_heart}
Let $X$ be a GM fourfold and $\sigma$ be a stability condition on $\Ku(X)$ and $\cA$ be its heart. Then \begin{enumerate}
    \item the homological dimension of $\cA$ is $2$.
    \item If $X$ is a non-Hodge-special GM fourfold, then we have $\mathrm{ext}^1(A,A)\geq 4$ for any non-trivial object $A\in\cA$.
\end{enumerate}
\end{proposition}

\begin{proof}
Let $A,B\in\cA$, then $\mathrm{Hom}(A,B[i])\cong\mathrm{Hom}(B[i],A[2])\cong\mathrm{Hom}(B,A[2-i])=0$ for $i\geq 3$. Thus $(1)$ holds. Let $A$ be a non-trivial object in $\cA$, then by $(1)$, $$\chi(A,A)=\mathrm{hom}(A,A)-\mathrm{ext}^1(A,A)+\mathrm{ext}^2(A,A).$$
Note that $\chi(A,A)\leq -2$. Then $\mathrm{ext}^1(A,A)=2\mathrm{hom}(A,A)-\chi(A,A)\geq 4$, which proves $(2)$. 
\end{proof}

By the same argument as in \cite[Lemma 2.4]{bayerK3}, we have

\begin{lemma} (Weak Mukai Lemma)
\label{mukai_lemma}
Let $A\rightarrow E\rightarrow B$ be a triangle in $\Ku(X)$ with $\mathrm{Hom}(A,B)=0$, then we have 
$$\mathrm{ext}^1(A,A)+\mathrm{ext}^1(B,B)\leq\mathrm{ext}^1(E,E).$$
\end{lemma}

Using Proposition \ref{homological_dimensional_heart} and  Lemma \ref{mukai_lemma}, the same argument as \cite[Lemma A.5]{bayer2017stability} shows that:

\begin{proposition}
\label{stability_object_ext1_small}
Let $X$ be a non-Hodge-special GM fourfold. If $E\in\Ku(X)$ is an object with $\mathrm{ext}^1(E,E)\leq 4$, then $E$ is stable with respect to every stability condition~$\sigma$ on $\Ku(X)$. 
\end{proposition}

\section{Conics on Gushel--Mukai fourfolds}\label{classfication_conics_GM}
In \cite{iliev2011fano}, the moduli space of conics $F_g(X)$ on a general ordinary GM fourfold $X$ plays an important role in the construction of the hyperkähler variety $\widetilde{Y}_{A^\perp}$.
In this section, we present some basic properties of conics in $F_g(X)$ and list some calculation results which heavily rely on the geometry of $X$.

According to \cite{debarre2015special}, there are two types of planes in $\Gr(2,V_5)$ which are called \emph{$\sigma$-planes} and \emph{$\rho$-planes}. A \emph{$\sigma$-plane} is in the form $\mathbb{P}(V_1\wedge V_4)$ and a \emph{$\rho$-plane} is in the form $\mathbb{P}(\wedge^2V_3)$ for some subspaces $V_1\subset V_4\subset V_5$ and $V_3\subset V_5$, respectively. Similarly, a \emph{$\sigma$-3-plane} is $\PP^3\subset\Gr(2,V_5)$ in the form $\PP(V_1\wedge V_4)$ for some subspaces $V_1,V_4\subset V_5$ with $V_1\nsubseteq V_4$. Note that in $\mathrm{Gr}(2,V_5)$, every $\PP^3$ is a $\sigma$-3-plane.

Recall that for a $n$-dimensional vector space $V_n$ and $r\geq 1$, the zero locus of a non-zero section of  $\cU^{\vee}_{\Gr(r,V_n)}$ is $\Gr(r,V_{n-1})$ for a codimension one subspace $V_{n-1}\subset V_n$, where $\cU_{\Gr(r,V_n)}$ is the tautological subbundle of $\Gr(r,V_n)$. Similarly, since the tautological quotient bundle $\cQ_{\Gr(r,V_n)}$ is isomorphic to the dual of the tautological subbundle of $\Gr(n-r, V_n^{\vee})$ via the isomorphism $\Gr(r,V_n)\cong \Gr(n-r,V_n^{\vee})$, then the zero locus of a non-zero section of $\cQ_{\Gr(r,V_n)}$ is $\Gr(r-1, V_n/V_1)\subset \Gr(r,V_n)$ for a one-dimensional subspace $V_1\subset V_n$.

Therefore, a $\sigma$-3-plane in $\Gr(2,V_5)$ is indeed the zero locus of a non-zero section of $\cQ_{\Gr(2,V_5)}$, which is isomorphic to $\Gr(1,3)$. Similarly, a $\rho$-plane is nothing but the zero locus of two linearly independent sections of $\cU^{\vee}_{\Gr(2,V_5)}$, which is isomorphic to $\Gr(2,3)$.

\begin{definition}[{\cite[Section 3.1]{iliev2011fano}}]\label{classification_of_conics}
Let $X$ be an ordinary GM fourfold, there are three types of conics
\begin{enumerate}
    \item $\tau$-conics are conics spanning planes which are not contained in $\mathrm{Gr}(2,V_5)$.
    \item $\sigma$-conics are conics parametrizing lines passing through a common point, i.e., they span $\sigma$-planes.
    \item $\rho$-conics are conics parametrizing lines contained in a common plane, i.e., they span $\rho$-planes.
\end{enumerate}
\end{definition}

\begin{lemma}[{\cite[Section  3.2]{iliev2011fano}}]\label{splitting_type}
Let $X$ be a general ordinary GM fourfold and $\cU, \cQ$ be the tautological sub and the quotient bundle on $X$. For a smooth conic $C$ on $X$, we have
\begin{enumerate}
    \item $\mathcal{U}^{\vee}|_C\cong\oh_C(1)\oplus\oh_C(1)$ and $\cQ|_C\cong\oh_C(1)\oplus\oh_C(1)\oplus\oh_C$ if $C$ is a $\tau$-conic;
    \item $\mathcal{U}^{\vee}|_C\cong\oh_C(2)\oplus\oh_C$ and $\mathcal{Q}|_C\cong\oh_C(1)\oplus\oh_C(1)\oplus\oh_C$ if $C$ is a $\sigma$-conic;
    \item $\mathcal{U}^{\vee}|_C\cong\oh_C(1)\oplus\oh_C(1)$ and $\cQ|_C\cong\oh_C(2)\oplus\oh_C\oplus\oh_C$ if $C$ is a $\rho$-conic.
\end{enumerate}
\end{lemma}

In fact, the type of a conic on $X$ is detected by the numbers $\mathrm{hom}(\cU, I_C)$ and $\mathrm{hom}(\cQ^{\vee}, I_C)$. 


\begin{lemma} \label{classify-conic}
Let $X$ be a general ordinary GM fourfold and $C$ be a conic on $X$.

\begin{enumerate}
    \item If $C$ is a $\tau$-conic, then $\Hom(\cU, I_C)=k$ and $\Hom(\cQ^{\vee}, I_C)=0$.
    
    \item If $C$ is a $\sigma$-conic, then $\Hom(\cU, I_C)=k$ and $\Hom(\cQ^{\vee}, I_C)=k$.
    
    \item If $C$ is a $\rho$-conic, then $\Hom(\cU, I_C)=k^2$ and $\Hom(\cQ^{\vee}, I_C)=0$.
\end{enumerate}

\end{lemma}

\begin{proof}
Note that for a conic $C$, if $\Hom(\cU, I_C)=k^a$, for some integer $a\geq 0$, then $C$ is contained in $\mathrm{Gr}(2, V_{5-a})\cap X$ by the discussion above. Since for any conic $C$, there is some $V_4$ such that $C$ lies in $\mathrm{Gr}(2,V_4)$, then we have $\hom(\cU, I_C)\geq 1$. Now if $\hom(\cU, I_C)\geq 2$, we know that $C$ is contained in a $\rho$-plane $\mathrm{Gr}(2,V_3)$. For a $\tau$-conic $C$,  $\langle C\rangle$ is not contained in $\mathrm{Gr}(2,V_4)$ for any $V_4\subset V_5$ and a $\sigma$-conic $C$ generates a $\sigma$-plane $\mathbb{P}(V_1\wedge V_4)$. Thus for such two types of conics, we have $\Hom(\cU, I_C)=k$. For a $\rho$-conic $C$, since $\langle C \rangle$ is in the form $\mathrm{Gr}(2,V_3)$, we have $\hom(\cU, I_C)\geq 2$. But if $\hom(\cU, I_C)\geq 3$, we know that $C\subset\mathrm{Gr}(2,V_2)$, which is impossible. Hence for a $\rho$-conic $C$, we have $\Hom(\cU, I_C)=k^2$.

On the other hand, if $\Hom(\cQ^{\vee}, I_C)=k^b$ for an integer $b\geq0$, $C$ is contained in $\mathrm{Gr}(2-b, V_{5-b})\cap X$. Thus we have $\hom(\cQ^{\vee}, I_C)\leq 1$ for any conic $C$. It is easy to see $\hom(\cQ^{\vee}, I_C)=1$ if and only if $C$ is contained in the zero locus of a global section of $\cQ$, which is a $\sigma$-3-plane $\mathbb{P}(V_1\wedge V_5)$ of $\mathrm{Gr}(2, V_5)$. This implies that $\Hom(\cQ^{\vee}, I_C)=0$ for conics of type $\tau$ or $\rho$ and $\Hom(\cQ^{\vee}, I_C)=k$ for conics of type $\sigma$.
\end{proof}

Now we state two lemmas, which are useful in the following sections.

\begin{lemma} \label{coho-1}
Let $X$ be an ordinary GM fourfold, then we have

\begin{enumerate}
    \item $\RHom(\cU, \cQ^{\vee})=k[0]$.

    \item $\RHom(\cU^{\vee}, \cQ)=0$.
    
    \item $\RHom(\cU^{\vee}, \cU)=0$.
    
    \item $\RHom(\cQ^{\vee}, \cQ(-H))=k[-2]$.
    
    \item $\RHom(\cU^{\vee}, \cQ^{\vee})=k[-1]$.

    \item $\RHom(\cU, \cQ^{\vee}(H))=k^{46}[0]$.

    \item $\RHom(\oh_X, \cU^{\vee}(H))=k^{35}[0]$.
    
\end{enumerate}

\end{lemma}

\begin{proof}
The Koszul resolution of an ordinary GM fourfold is in the form
\[0\to \oh_{\mathrm{Gr}(2, V_5)}(-3)\to \oh_{\mathrm{Gr}(2, V_5)}(-2)\oplus \oh_{\mathrm{Gr}(2, V_5)}(-1)\to \oh_{\mathrm{Gr}(2, V_5)}\to \oh_X\to 0.\]
Then the result follows from this resolution and a standard computation applying the Borel--Weil--Bott theorem, see e.g.~\cite[(4.1.9), (4.1.12)]{weyman_2003}.
\end{proof}

\begin{lemma}  \label{coho-2}
Let $X$ be a general ordinary GM fourfold and $C\subset X$ be a conic. Then we have
\begin{enumerate}
    \item $\RHom(\oh_X(H), \oh_C)=k[-1]$.
    
    \item $\RHom(\cU^{\vee}, \oh_C)=0$ when $C$ 
 is of type $\tau$ or $\rho$.
  
  \item $\RHom(\cU^{\vee}, \oh_C)=k[0]\oplus k[-1]$ when $C$ is a $\sigma$-conic.
 
    \item $\RHom(\cU^{\vee}(H), \oh_C)=k^4[-1]$.
    
    \item $\RHom(\cQ^{\vee}(H), \oh_C)=k[-1]$ when $C$ is of type $\tau$ or $\sigma$.
    
    \item $\RHom(\cQ^{\vee}(H), \oh_C)=k[0]\oplus k^2[-1]$ when $C$ is of type $\rho$.
\end{enumerate}

\end{lemma}

\begin{proof}
When $C$ is smooth, the result follows from $C\cong \mathbb{P}^1$ and Lemma \ref{splitting_type}. 

When $C$ is not smooth but reduced, we have an exact sequence
\[0\to \oh_C\to \oh_{l_1}\oplus \oh_{l_2}\to \oh_x\to 0,\]
where $l_i$ are lines such that $l_1\cup l_2=C$, $l_1\cap l_2=x$. When $C$ is non-reduced, we have
\[0\to \oh_l(-H)\to \oh_C\to \oh_l\to 0,\]
where $l=C_{red}$ is a line. Then the results follow from applying the Hom-functor to two exact sequences above, taking long exact sequences and applying Lemma \ref{classify-conic}.
\end{proof}

Finally, we introduce a surface $q$, which is called a \emph{$\sigma$-quadric surface}. We follow \cite[Section 3]{debarre2015special}. 

Recall that $X=\mathrm{Gr}(2,V_5)\cap \PP^8\cap Q$. The hyperplane $\PP^8$ is defined by a nonzero skew-symmetric form $\omega$ on $V_5$. Then the $\sigma$-quadric surface $q$ is defined by a $\sigma$-3-plane $\mathbb{P}(V_1^\omega\wedge V_5)$ in $\mathrm{Gr}(2,V_5)$ intersecting with $X$, where $V_1^\omega\subset V_5$ is the kernel of the nonzero skew-symmetric form $\omega$ on $V_5$. Indeed, $\langle q\rangle=\mathbb{P}(V_1^\omega\wedge V_5)$, i.e.~$q$ is the zero locus of a section of $\cQ$.
Moreover, when $X$ is general, $q$ is smooth and is the only quadric surface contained in $X$ due to \cite[Remark 2.2]{perry2019stability} and  \cite{debarre2015special} .

\begin{proposition}\label{sequence_of_Iq}
   Let $I_q$ be the ideal sheaf of the $\sigma$-quadric $q$. Then there is an exact sequence
   $$0\to\cU\to\cQ^{\vee}\to I_q\to0.$$
\end{proposition}

\begin{proof}
By definition, $q$ is the zero locus of a section of $\cQ$, which corresponds to a surjective map $\pi_q\colon\cQ^{\vee}\twoheadrightarrow I_q$.

It is left to show that $\ker(\pi_q)\cong\cU$. 
Since $\cU$ is slope-stable, we only need to prove that $\ker(\pi_q)$ is slope-stable as well and $\Hom(\cU,\ker(\pi_q))\neq 0$. If $\ker(\pi_q)$ is not stable, we choose a destabilizing sheaf $D$ of $\ker(\pi_q)$. Then we have $\mu(\cQ^{\vee})=-\frac{1}{3}>\mu(D)\geq\mu(K)=-\frac{1}{2}$. While $\Pic(X)=\mathbb{Z}.H$ and $\rk(D)=1$, this can not happen. Thus $\ker(\pi_q)$ is slope-stable. Then we apply $\Hom(\cU,-)$ to the exact sequence
$$0\to \ker(\pi_q)\to \cQ^{\vee}\to I_q\to 0.$$
By Lemma~\ref{coho-1}, $\Hom(\cU,\cQ^\vee)= k$. At the same time, $\Hom(\cU,I_q)=0$, otherwise $q$ is contained in the zero locus of a section of $\cU^{\vee}$, which is $\Gr(2,V_4)\cap H\cap Q$ for some $V_4\subset V_5$ and can not happen. Then we get $\Hom(\cU, \ker(\pi_q))= k$ as desired.
\end{proof}

\section{Projection objects of conics}\label{projection_objects}

Recall that $\mathrm{pr}_1:=\bL_{\oh_X}\bL_{\mathcal{U}^{\vee}}\bL_{\oh_X(H)}\bL_{\mathcal{U}^{\vee}(H)}$ and $\mathrm{pr}_2:=\bR_{\cU}\bR_{\oh_X(-H)}\bL_{\oh_X}\bL_{\mathcal{U}^{\vee}}$. In this section, we find out the objects $\pr_1(\oh_C(H))$ and $\pr_2(I_C)$ for any conic $C$ on a general ordinary GM fourfold, then we relate them via the involution $T$ defined in Section~\ref{semi-orthogonal_GM}.

The characters $\mathrm{ch}(\mathrm{pr}_1(\oh_C(H)))$ and $\ch(\pr_2(I_C))$ of a conic $C$ are $\Lambda_1$. We start with two lemmas.

\begin{lemma} \label{pr2-3fold}
Let $X$ be an ordinary GM fourfold and $C$ be a conic on $X$. Let $j:Y\hookrightarrow X$ be any hyperplane section containing $C$. Then we have
\[\pr_2(I_C)\cong\pr_2(I_{C/Y}).\]
\end{lemma}

\begin{proof}
Note that we have an exact sequence
\[0\to \oh_X(-H)\to I_C\to I_{C/Y}\to 0.\]
Since $\bL_{\oh_X}\bL_{\cU^{\vee}}(\oh_X(-H))=\oh_X(-H)$, thus we have $$\pr_2(\oh_X(-H))=\bR_{\cU}\bR_{\oh_X(-H)}\bL_{\oh_X}\bL_{\mathcal{U}^{\vee}}(\oh_X(-H))=0.$$ 
Therefore we obtain $\pr_2(I_C)\cong\pr_2(I_{C/Y})$.
\end{proof}

Recall that by dualizing the exact triangle defining the mutation functor, we have $\mathbb{D}(\bL_{E})\cong \bR_{\mathbb{D}(E)}\circ \mathbb{D}$ and $\mathbb{D}(\bR_{E})\cong \bL_{\mathbb{D}(E)}\circ \mathbb{D}$ for any object $E$.

\begin{lemma} \label{T-action}
Let $X$ be a general ordinary GM fourfold and $C$ be a conic on $X$. Then we have
\[T(\pr_2(I_C))\cong \pr_1(\oh_C(H))[-2].\]
\end{lemma}

\begin{proof}
By definition, we have
\[T(\pr_2(I_C))=\bL_{\oh_X}\mathbb{D}(\bR_{\cU}\bR_{\oh_X(-H)}\bL_{\oh_X}\bL_{\mathcal{U}^{\vee}}I_C).\]
Since $\mathbb{D}\circ (\bR_{\cU}\bR_{\oh_X(-H)}\bL_{\oh_X}\bL_{\mathcal{U}^{\vee}})=\bL_{\cU^{\vee}}\bL_{\oh_X(H)}\bR_{\oh_X}\bR_{\cU}\circ \mathbb{D}$, we get 
\[T(\pr_2(I_C))=\bL_{\oh_X}\bL_{\cU^{\vee}}\bL_{\oh_X(H)}\bR_{\oh_X}\bR_{\cU}(\mathbb{D}(I_C)).\]
Applying $\bR_{\oh_X}\bR_{\cU}(\mathbb{D}(-))$ to the standard exact sequence of $C$, we obtain 
\[\bR_{\oh_X}\bR_{\cU}(\mathbb{D}(I_C))\cong \bR_{\oh_X}\bR_{\cU}(\mathbb{D}(\oh_C))[1],\]
which means
\[T(\pr_2(I_C))\cong\bL_{\oh_X}\bL_{\cU^{\vee}}\bL_{\oh_X(H)}\bR_{\oh_X}\bR_{\cU}(\mathbb{D}(\oh_C))[1].\]

Firstly, we deal with $\tau$-conics and $\rho$-conics. According to Lemma \ref{coho-2}(2), we have $$\bL_{\oh_X(H)}\bL_{\cU^{\vee}(H)}(\oh_C(H))=I_C(H)[1],$$
\[\bR_{\oh_X}\bR_{\cU}(\oh_C(H))=\bR_{\oh_X}(\oh_C(H))\]
and a triangle 
\[\bR_{\oh_X}\bR_{\cU}(\oh_C(H))\to \oh_C(H)\to \oh_X[3].\]
Applying $\bL_{\oh_X}\bL_{\cU^{\vee}}\bL_{\oh_X(H)}$ to this triangle, we obtain that
\[\bL_{\oh_X}\bL_{\cU^{\vee}}\bL_{\oh_X(H)}\bR_{\oh_X}\bR_{\cU}(\oh_C(H))\cong\bL_{\oh_X}\bL_{\cU^{\vee}}\bL_{\oh_X(H)}(\oh_C(H))\cong\bL_{\oh_X}\bL_{\cU^{\vee}}(I_C(H)[1]).\]
The last complex is just $\pr_1(\oh_C(H))$. In order to establish $T(\pr_2(I_C))\cong \pr_1(\oh_C(H))[-2]$, we only need to show that $\mathbb{D}(\oh_C)\cong\oh_C(H)[-3]$.
\begin{itemize}
    \item Let $C$ be a smooth conic of type $\tau$ or $\rho$. By Grothendieck--Verdier duality, we have $\mathbb{D}(\oh_C)\cong\oh_C(H)[-3]$.

\item Let $C$ be a non-smooth reduced conic of type $\tau$ or $\rho$. We have the exact sequence
\[0\to \oh_C\to \oh_{l_1}\oplus \oh_{l_2}\to \oh_x\to 0\]
where $l_i$ are lines such that $l_1\cup l_2=C$, $l_1\cap l_2=x$. By Grothendieck--Verdier duality, we have $\mathbb{D}(\oh_{l_i})\cong \oh_{l_i}[-3]$ and $\mathbb{D}(\oh_x)\cong \oh_x[-4]$. Thus applying $\mathbb{D}$ to the exact sequence above, we obtain a triangle
\[\oh_x[-4]\to (\oh_{l_1}\oplus \oh_{l_2})[-3]\to \mathbb{D}(\oh_C).\]
It means we have an exact sequence
\[0\to \oh_{l_1}\oplus \oh_{l_2}\to \mathbb{D}(\oh_C)[3]\to \oh_x\to 0,\]
then we obtain $\mathbb{D}(\oh_C)\cong\oh_C(H)[-3]$.

\item Let $C$ be a double line of type  $\tau$ or $\rho$. We have
\[0\to \oh_l(-H)\to \oh_C\to \oh_l\to 0,\]
where $l=C_{red}$ is a line. Applying $\mathbb{D}$ to this exact sequence, we have a triangle
\[\oh_l[-3]\to \mathbb{D}(\oh_C)\to \oh_l(H)[-3],\]
then we obtain $\mathbb{D}(\oh_C)\cong\oh_C(H)[-3]$.
\end{itemize}

When $C$ is a $\sigma$-conic, the computation is similar. We omit details here.
\end{proof}

Next, we compute projection objects of all three types of conics to the Kuznetsov component $\Ku(X)$.

\subsection{$\tau$-conic}

In this subsection, we compute the projection objects of $\tau$-conics.

\begin{proposition} \label{proj-tau}
Let $X$ be a general ordinary GM fourfold and $C$ be a $\tau$-conic on $X$. Then we have
\[\pr_1(\oh_C(H))\cong\bL_{\oh_X}(I_{C/\Sigma}(H))[1],\]
where $\Sigma$ is the zero locus of a section of $\cU^{\vee}$ containing $C$. Moreover, there is an exact sequence
\[0\to \cU^{\oplus 4}\to K_1\to \pr_2(I_C)\to I_C\to 0,\]
where $K_1:=\mathrm{cok}(\oh_X(-H)\hookrightarrow \cU^{\oplus 5})=\bR_{\cU}\oh_X(-H)[1]$.
\end{proposition}

\begin{proof}
By definition of $\pr_1$, we have
\begin{align*}
\mathrm{pr}_1(\oh_C(H))&=\bL_{\oh_X}\bL_{\cU^{\vee}}\bL_{\oh_X(H)}\bL_{\cU^{\vee}(H)}(\oh_C(H))\\
&\cong\bL_{\oh_X}\bL_{\cU^{\vee}}\bL_{\oh_X(H)}(\oh_C(H))\\
&\cong\bL_{\oh_X}\bL_{\cU^{\vee}}(I_C(H))[1].
\end{align*}
The first isomorphism follows from Lemma \ref{coho-2}. The second isomorphism follows from the standard exact sequence associated with $C$. Next, 
we have an exact triangle 
$$\mathrm{RHom}(\cU^{\vee},I_C(H))\otimes\cU^{\vee}\rightarrow I_C(H)\rightarrow\bL_{\cU^{\vee}}(I_C(H)).$$
Note that $\mathrm{RHom}(\cU^{\vee},I_C(H))\cong\mathrm{RHom}(\cU,I_C)$. By Lemma \ref{classify-conic}, $\mathrm{RHom}(\cU,I_C)=k[0]$. Then the object $\bL_{\cU^{\vee}}(I_C(H))$ fits into the triangle
$$\cU^{\vee}\rightarrow I_C(H)\rightarrow\bL_{\cU^{\vee}}(I_C(H)).$$
The image of $\pi:\cU\rightarrow I_C$ is the ideal sheaf $I_{\Sigma}$, where $\Sigma$ is the zero locus of a section of $\cU^{\vee}$ containing $C$, which is a surface $\Sigma=\mathrm{Gr}(2,V_4)\cap X$ for some $V_4$. Then we have two short exact sequences
$$0\rightarrow\ker\pi\rightarrow\cU\rightarrow I_{\Sigma}\rightarrow 0$$ and $$0\rightarrow I_{\Sigma}\rightarrow I_C \rightarrow I_{C/\Sigma}\rightarrow 0.$$
Note that $\ker\pi$ is a rank one reflexive sheaf on $X$, hence is a line bundle. 
Thus we have $\ker\pi\cong \oh_X(-H)$ and there is an exact triangle
$$\oh_X(-H)[2]\rightarrow\bL_{\cU}(I_C)[1]\rightarrow I_{C/\Sigma}[1].$$
Tensoring with $\oh_X(H)$, we get
$$\oh_X[2]\rightarrow\bL_{\cU^{\vee}}(I_C(H))[1]\rightarrow I_{C/\Sigma}(H)[1].$$
Finally, applying $\bL_{\oh_X}$ to this triangle, we get 
$\mathrm{pr}_1(\oh_C(H))\cong\bL_{\oh_X}(I_{C/\Sigma}(H))[1]$.

Now we compute $\pr_2(I_C)$. Since $\RHom(\oh_X, I_C)=\RHom(\cU^{\vee},I_C)=0$ by Lemma \ref{coho-2}, we have
\[\pr_2(I_C)=\bR_{\cU}\bR_{\oh_X(-H)}(I_C).\]
Since we have $\RHom(I_C, \oh_X(-H))=k[-2]$ and $\RHom(I_C, \cU)=k^4[-2]$, then we obtain triangles
\[\bR_{\oh_X(-H)}(I_C)\to I_C\to \oh_X(-H)[2],\]
\[\bR_{\cU}(I_C)\to I_C\to \cU^{\oplus 4}[2]\]
and
\[\bR_{\cU}\bR_{\oh_X(-H)}I_C\to \bR_{\cU}I_C\to K_1[1],\]
where $K_1:=\mathrm{cok}(\oh_X(-H)\hookrightarrow \cU^{\oplus 5})=\bR_{\cU}\oh_X(-H)[1]$. Therefore, taking the long exact sequence of cohomology, we get
\begin{equation}\label{sequence of alpha}
0\to \cH^{-1}(\pr_2(I_C))\to \cU^{\oplus 4}\xra{\alpha} K_1 \to \cH^0(\pr_2(I_C))\to I_C \to 0.    
\end{equation}

\textbf{Claim.} We have $\ker(\alpha)=\cH^{-1}(\pr_2(I_C))=0$. Thus the sequence of (\ref{sequence of alpha}) becomes
\[0\to \cU^{\oplus 4}\xra{\alpha} K_1 \to \pr_2(I_C)\to I_C \to 0.\]

Now the rest of the proof aims to prove this claim. From the definition of $\alpha$, we have a commutative diagram
\[\begin{tikzcd}
	{\Hom(I_C[-2], \cU)^{\vee}\otimes \cU[1]} & {\bR_{\cU}I_C} & {I_C} \\
	{\Hom(\oh_X(-H), \cU)^{\vee}\otimes \cU[1]} & {\bR_{\cU}\oh_X(-H)[2]} & {\oh_X(-H)[2]}
	\arrow["v", from=1-3, to=2-3]
	\arrow[from=1-2, to=1-3]
	\arrow[from=2-2, to=2-3]
	\arrow[from=2-1, to=2-2]
	\arrow["{f_2}"', from=1-1, to=1-2]
	\arrow["{v'}", from=1-2, to=2-2]
	\arrow["{f_1}", from=1-1, to=2-1]
\end{tikzcd}.\]
Here $v'\circ f_2=\alpha[1]$ and rows are induced by the definition of functor $\bR_{\cU}$. Let $v$ be a non-zero element in $\Hom(I_C, \oh_X(-H)[2])=k$, $v'$ and $f_1$ be the morphisms induced by $v$ and the right mutation functor. To determine $f_1$, we only need to determine the natural map $f_3: \Hom(I_C[-2], \cU)^{\vee}\to \Hom(\oh_X(-H), \cU)^{\vee}$ induced by $v$ due to the fact $f_3\otimes \mathrm{id}_{\cU[1]}=f_1$. 

To this end, using Serre duality we see that $f_3\colon \Hom(I_C[-2], \cU)^{\vee}\to \Hom(\oh_X(-H), \cU)^{\vee}$ is actually the dual map of $f_4\colon \Hom(\oh_X(-H),\cU)\to \Hom(I_C[-2], \cU)$ induced by $v\colon I_C\to \oh_X(-H)[2]$. We claim that $f_3$ is injective, which implies $\ker(\alpha)=0$. Indeed, we only need to show $f_4$ is surjective, which is equivalent to show $\Hom(\bR_{\oh_X(-H)}(I_C), \cU)=0$ since $\bR_{\oh_X(-H)}(I_C)=\mathrm{cone}(v)[-1]$.

By Serre duality and adjunction of mutation functors, we have
\[\Hom(\bR_{\oh_X(-H)}(I_C), \cU)=\Ext^4(\cU^{\vee}(H),\bR_{\oh_X(-H)}(I_C))=\Ext^4(\bL_{\oh_X(-H)}\cU^{\vee}(H),I_C).\]
Note that by (7) of Lemma \ref{coho-1} and Serre duality, we see $\cU^{\vee}(2H)$ is regular in the sense of Castelnuovo--Mumford, i.e.~$H^i(\cU^{\vee}((2-i)H))=0$ for any $i\geq 1$. Hence $\cU^{\vee}(2H)$ is globally generated, which implies $\bL_{\oh_X(-H)}\cU^{\vee}(H)[-1]$ is a vector bundle. Moreover, we have $\Ext^i(\bL_{\oh_X(-H)}\cU^{\vee}(H), \oh_X)=0$ for $i\geq 2$ by applying $\Hom(-,\oh_X)$ to the exact triangle defining $\bL_{\oh_X(-H)}\cU^{\vee}(H)$. Then applying $\Hom(\bL_{\oh_X(-H)}\cU^{\vee}(H),-)$ to the exact sequence $0\to I_C\to \oh_X\to \oh_C\to 0$, we see
\[\Ext^4(\bL_{\oh_X(-H)}\cU^{\vee}(H),I_C)=\Ext^3(\bL_{\oh_X(-H)}\cU^{\vee}(H),\oh_C)=0.\]



\end{proof}

\subsection{Projection of $\cU$}\label{section_projection_of_tauto}

In this subsection, we find out the projection object of $\cU$.

\begin{proposition}\label{projection_tauto}
Let $X$ be a general ordinary GM fourfold. Then we have a triangle
\[\cU\to \pr_1(\cU)\to K_2[-1],\]
where $K_2:=\bL_{\oh_X}(I_q(H))[-1]$ is a $\mu$-stable reflexive sheaf, $q$ is the unique $\sigma$-quadric in $X$.
\end{proposition}

\begin{proof}

We apply the first mutation $\bL_{\cU^{\vee}(H)}$ to $\cU$,
$$\mathrm{RHom}(\cU^{\vee}(H),\cU)\otimes\cU^{\vee}(H)\rightarrow\cU\rightarrow\bL_{\cU^{\vee}(H)}\cU.$$

By Serre duality and exceptionality of $\cU$, we have $\RHom(\cU^{\vee}(H), \cU)=k[-4]$. Then the triangle becomes
$$\cU^{\vee}(H)[-4]\rightarrow\cU\rightarrow\bL_{\cU^{\vee}(H)}\cU.$$
Applying $\bL_{\oh_X(H)}$ to this triangle, we get 
$$\bL_{\oh_X(H)}\cU^{\vee}(H)[-4]\rightarrow\bL_{\oh_X(H)}\cU\rightarrow\bL_{\oh_X(H)}\bL_{\cU^{\vee}(H)}\cU.$$
By Serre duality, $\RHom(\oh_X(H), \cU)=0$, so that $\bL_{\oh_X(H)}\cU\cong\cU$. Then we have the triangle
$$\cQ^{\vee}(H)[-3]\rightarrow\cU\rightarrow\bL_{\oh_X(H)}\bL_{\cU^{\vee}(H)}\cU.$$
Applying $\bL_{\cU^{\vee}}$, we obtain
$$(\bL_{\cU}\cQ^{\vee})(H)[-3]\rightarrow\bL_{\cU^{\vee}}\cU\rightarrow\bL_{\cU^{\vee}}\bL_{\oh_X(H)}\bL_{\cU^{\vee}(H)}\cU.$$

Since $\RHom(\cU^{\vee}, \cU)=0$ by Lemma \ref{coho-1}, we have the following triangle
$$(\bL_{\cU}\cQ^{\vee})(H)[-3]\rightarrow\cU\rightarrow\bL_{\cU^{\vee}}\bL_{\oh_X(H)}\bL_{\cU^{\vee}(H)}\cU.$$
Now from the fact $\RHom(\cU, \cQ^{\vee})=k[0]$, there is a triangle
$$\cU\rightarrow\cQ^{\vee}\rightarrow\bL_{\cU}\cQ^{\vee}.$$ 
Since $\cU$ and $\cQ^{\vee}$ are both $\mu$-stable with slopes $\mu(\cU)=-\frac{1}{2}$ and $\mu(\cQ^{\vee})=-\frac{1}{3}$, the map $\cU\xrightarrow{s}\cQ^{\vee}$ is injective and $\bL_{\cU}\cQ^{\vee}\cong\mathrm{cok}(s)\cong I_q$ by Proposition~\ref{sequence_of_Iq}. 
Then the triangle becomes 
$$I_q\otimes\oh_X(H)[-3]\rightarrow\cU\rightarrow\bL_{\cU^{\vee}}\bL_{\oh_X(H)}\bL_{\cU^{\vee}(H)}\cU.$$
Applying $\bL_{\oh_X}$, we get 
$$\bL_{\oh_X}(I_q\otimes\oh_X(H))[-3]\rightarrow\cU\rightarrow\pr_1(\cU).$$

As in Section~\ref{classification_of_conics}, $q=\langle q\rangle\cap X=\langle q\rangle\cap Q$, which means $q$ is cut out by five hyperplane sections of $\mathbb{P}^8$. This implies the morphism $t: \mathcal{O}_X^{\oplus5}\longrightarrow I_q(H)$ is surjective and $\bL_{\mathcal{O}_X}(I_q(H))[-1]\cong\ker(t)$, which is denoted by $K_2$. By \cite[Proposition 1.1]{har80}, $K_2$ is a reflexive sheaf. Finally, the stability of $K_2$ follows from the poly-stability of $\oh^{\oplus 5}_X$ and the fact $\RHom(\oh_X, K_2)=0$.
\end{proof}

\subsection{$\rho$-conic}

In this subsection, we compute the projection objects of $\rho$-conics. 

At first, we offer two lemmas which will be very useful in the proof of Proposition~\ref{pr2-rho}.

\begin{lemma} \label{no_sheaf_T}
Let $X$ be an ordinary GM fourfold and $F$ be a $\mu$-semistable sheaf on $X$ with $\rk(F)=3$, $\ch_1(F)=-H$ and $H\cdot \ch_2(F)=eL$. Then we have $e\leq -1$.
\end{lemma}

\begin{proof}
By Mayamura's restriction theorem, we can take a general smooth hyperplane section $Y$ such that $F|_Y$ remains $\mu$-semistable. Then $\ch_{\leq 2}(F|_Y)=(3,-H, eL)$. The result follows from  \cite[Proposition 3.2]{Li15}.
\end{proof}

\begin{lemma} \label{coho_rho_prop}
Use the notations as in Proposition~\ref{projection_tauto}, we have

\begin{enumerate}
    \item $\RHom(I_q(H), \cU)=k[-3]$.
     
    \item $\RHom(K_2,\cU)=k[-2]$.
    
    \item $\Ext^1(I_q(H), I_C)=k$.
    
    \item $\Hom(K_2, I_C)=k$.
\end{enumerate}

\end{lemma}

\begin{proof}
For $(1)$, by Serre duality we only need to compute $\RHom(\cU^{\vee}, I_q)$. Then we apply $\Hom(\cU^{\vee}, -)$ to the exact sequence
\[0\to \cU\to \cQ^{\vee} \to I_q\to 0,\]
the result follows from Lemma \ref{coho-1}.

For $(3)$, we apply $\Hom(-, I_C)$ to the exact sequence
\[0\to \cU(H)\to \cQ^{\vee}(H) \to I_q(H)\to 0,\]
since $C$ is a $\rho$-conic, the result follows from $\RHom(\cU^{\vee}, I_C)=0$ in Lemma~\ref{classify-conic} and Lemma~\ref{coho-2}.

Now if we apply $\Hom(-,\cU)$ to the exact sequence
\[0\to K_2\to \oh_X^{\oplus 5}\to I_q(H)\to 0,\]
then $(2)$ follows from $(1)$. If we apply $\Hom(-, I_C)$ to the exact sequence above, we have $\Hom(K_2, I_C)\cong\Ext^1(I_q(H), I_C)$, then $(4)$ follows from $(3)$.
\end{proof}

\begin{proposition} \label{pr2-rho}
Let $X$ be a general ordinary GM fourfold and $C$ be a $\rho$-conic on $X$. Then we have
\[\pr_2(I_C)\cong\pr_1(\cU)[1].\]
\end{proposition}

\begin{proof}
As in the proof of Proposition \ref{proj-tau}, we have a long exact sequence
\begin{equation}
    0\to \cH^{-1}(\pr_2(I_C))\to \cU^{\oplus 4}\xra{\alpha} K_1 \to \cH^0(\pr_2(I_C))\to I_C \to 0.
\end{equation}
Note that in this case $\Hom(\cQ^{\vee}(H), \oh_C)=k$ by Lemma \ref{coho-2}. Thus $$\cH^{-1}(\pr_2(I_C))\cong\ker(\alpha)\cong\cU.$$
And we have an exact sequence
\begin{equation}\label{H0_seq}
    0\to K_3\to \cH^0(\pr_2(I_C))\to I_C\to 0,
\end{equation}
where $0\to \cU^{\oplus 3}\to K_1\to K_3\to 0.$ Note that we have a commutative diagram

\[\begin{tikzcd}
	&& 0 \\
	&& {\oh_X(-H)} \\
	0 & {\cU^{\oplus 3}} & {\cU^{\oplus 5}} & {\cU^{\oplus 2}} & 0 \\
	0 & {\cU^{\oplus 3}} & {K_1} & K_3 & 0 \\
	&& 0
	\arrow[from=4-1, to=4-2]
	\arrow[from=4-2, to=4-3]
	\arrow[from=4-3, to=4-4]
	\arrow[from=4-4, to=4-5]
	\arrow[from=4-3, to=5-3]
	\arrow[from=1-3, to=2-3]
	\arrow[from=2-3, to=3-3]
	\arrow[from=3-1, to=3-2]
	\arrow[from=3-2, to=3-3]
	\arrow[from=3-3, to=3-4]
	\arrow[from=3-4, to=3-5]
	\arrow[shift right=1, no head, from=3-2, to=4-2]
	\arrow[from=3-4, to=4-4]
	\arrow[from=3-3, to=4-3]
	\arrow[no head, from=3-2, to=4-2]
\end{tikzcd}\]

Hence $K_3$ also fits into an  exact sequence
\begin{equation}\label{T_seq_1}
    0\to \oh_X(-H)\to \cU^{\oplus 2}\to K_3\to 0.
\end{equation}
By the slope stability of $\cU$, we see that the torsion part of $K_3$ is supported in codimension $\geq 2$.  Since $K_3$ is a quotient of two bundles, we know that the torsion part of $K_3$ is zero or has pure codimension one, which implies the torsion-freeness of $K_3$, hence $\cH^0(\pr_2(I_C))$ is also torsion-free. 
From $\RHom(\cU, \oh_X(-H))=0$, we have the following commutative diagram
\[\begin{tikzcd}
	& 0 & \cU & \cU \\
	0 & {\oh_X(-H)} & {\cU^{\oplus 2}} & {K_3} & 0
	\arrow[from=2-1, to=2-2]
	\arrow[from=2-2, to=2-3]
	\arrow[from=2-3, to=2-4]
	\arrow[from=2-4, to=2-5]
	\arrow[hook, from=1-3, to=2-3]
	\arrow[from=1-2, to=2-2]
	\arrow[from=1-2, to=1-3]
	\arrow[shift left=1, no head, from=1-3, to=1-4]
	\arrow["\kappa"', from=1-4, to=2-4]
	\arrow[no head, from=1-3, to=1-4]
\end{tikzcd}\]
with rows exact. Note that $\kappa$ is injective. Indeed, if $\kappa$ is not injective, then $\rk(\im(\kappa))=1$. By the stability of $\cU$, we see $\ch_1(\im(\kappa))=xH$, where $x\geq 0$. Hence, we see $\rk(\mathrm{cok}(\kappa))=2$ and  $\ch_{1}(\mathrm{cok}(\kappa))=(-1-x)H$. Since $\mathrm{cok}(\kappa)$ is a quotient of $\cU^{\oplus 2}$, by the stability of $\cU$ again, we have $x=0$. But by the uniqueness of Jordan--Holder factors of $\cU^{\oplus 2}$, we know that $\mathrm{cok}(\kappa)\cong \cU$, which is impossible since $K_3$ is a quotient of $K_1=\bR_{\cU}\oh_X(-H)[1]$ (cf.~Proposition \ref{proj-tau}).

Now by the injectivity of $\kappa$ and the snake lemma, we get an exact sequence $0\to \oh_X(-H)\to \cU\to \mathrm{cok}(\kappa)\to 0$, which implies $\mathrm{cok}(\kappa)=I_{\Sigma_1}$, where $\Sigma_1$ is the zero locus of a regular section of $\cU^\vee$. Hence we have an exact sequence
\begin{equation} \label{T_seq_2}
    0\to \cU \xra{\kappa} K_3\to I_{\Sigma_1}\to 0,
\end{equation}

First, we claim that $K_3$ is $\mu$-(semi)stable. Indeed, by the stability of $\cU$ and $I_{\Sigma_1}$, the only possible case is that the maximal destabilizing subsheaf of $K_3$ is in the form $I_W$, where~$W$ is a closed subscheme containing the surface  $\Sigma_1$. It is easy to see $\rk (K_3/I_W)=2$ and $\ch_1(K_3/I_W)=-H$. Thus $(K_3/I_W)^{\vee \vee}$ is also $\mu$-semistable with rank two and $\ch_1=-H$. If we apply $\Hom(-,(K_3/I_W)^{\vee \vee})$ to the sequence (\ref{T_seq_2}), we have $\Hom(\cU, (K_3/I_W)^{\vee \vee})\neq 0$. By the stability of $\cU$ and $(K_3/I_W)^{\vee \vee}$, we know that $\cU\subset (K_3/I_W)^{\vee \vee}$, but this is impossible. This is because $\cU$ is locally free and $(K_3/I_W)^{\vee \vee}$ is reflexive, the support of the quotient
is of codimension $\leq 1$, which contradicts with the fact $(K_3/I_W)^{\vee \vee}/\cU$ is supported in codimension $\geq 2$.
 
Next we claim that $K_3$ is reflexive. Indeed, we have a commutative diagram 
\[\begin{tikzcd}
	0 & \cU & K_3 & {I_{\Sigma_1}} & 0 \\
	0 & \cU & {K_3^{\vee \vee}} & {\oh_X} & {}
	\arrow[from=1-1, to=1-2]
	\arrow[from=1-2, to=1-3]
	\arrow[from=1-3, to=1-4]
	\arrow[from=1-4, to=1-5]
	\arrow[from=2-1, to=2-2]
	\arrow[from=2-2, to=2-3]
	\arrow["\theta", from=2-3, to=2-4]
	\arrow[shift right=1, no head, from=1-2, to=2-2]
	\arrow[shift right=1, hook, from=1-3, to=2-3]
	\arrow[hook, from=1-4, to=2-4]
	\arrow[no head, from=1-2, to=2-2]
\end{tikzcd}\]
and $\Im(\theta)=I_{Z_1}$, where $Z_1$ is a closed subscheme contained in $\Sigma_1$. If $Z_1\neq \Sigma_1$, we can assume that $\ch(I_{Z_1})=1-eL+fP$, where $e\geq 0$. In this case $K_3^{\vee \vee}$ is also $\mu$-semistable. But $H\cdot \ch_2(K_3^{\vee \vee})=L$ and this contradicts Lemma \ref{no_sheaf_T}. Thus $Z_1=\Sigma_1$ and we know that $K_3\cong K_3^{\vee \vee}$.

Then we claim that $\cH^0(\pr_2(I_C))$ is $\mu$-(semi)stable. Indeed, if $\cH^0(\pr_2(I_C))$ is not $\mu$-semistable, let $K_4$ be its minimal destabilizing quotient sheaf. Then by (\ref{H0_seq}) and the stability of $K_3$ and $I_C$, it is not hard to see that the only possible case is $\rk(K_4)=3$ and $\ch_1(K_4)=-H$. Then if we apply $\Hom(-, K_4^{\vee \vee})$ to the triangle (\ref{H0_seq}), we obtain $\Hom(K_3, K_4^{\vee \vee})\neq 0$. Since they have the same rank and $\ch_1$, by stability we have $K_3\subset K_4^{\vee \vee}$, which is impossible since they are both reflexive but $K_4^{\vee \vee}/K_3$ is supported in codimension $\geq 2$.

Finally, we show that $\Hom(K_2, \cH^0(\pr_2(I_C)))=k$, then using the $\mu$-stability of $K_2$ and $\cH^0(\pr_2(I_C))$, we obtain $K_2\cong \cH^0(\pr_2(I_C))$. From the definition of $K_2$, it is not hard to see that $\RHom(K_2, \oh_X(-H))=0$. Now applying $\Hom(K_2, -)$ to the exact sequence (\ref{T_seq_1}), we obtain $\RHom(K_2, \cU^{\oplus 2})=\RHom(K_2, K_3)$. By Lemma~\ref{coho_rho_prop}, we know that $\RHom(K_2, K_3)=k^2[-2]$. Therefore, if we apply $\Hom(K_2, -)$ to the exact sequence (\ref{H0_seq}), we obtain 
$$\Hom(K_2, \cH^0(\pr_2(I_C)))=\Hom(K_2, I_C),$$ which equals to $k$ by Lemma \ref{coho_rho_prop}.

Recall that $\pr_1(\cU)[1]$ sits in the triangle 
\[\cU[1]\rightarrow\pr_1(\cU)[1]\rightarrow K_2.\]
Now we have established $\pr_2(I_C)$ and $\pr_1(\cU)[1]$ share the same cohomology objects. Then the result $\pr_2(I_C)\cong\pr_1(\cU)[1]$ follows from the fact  $\Ext^1(K_2,\cU[1])=k$.
\end{proof}



\subsection{$\sigma$-conic}

In this subsection, we compute the projection objects of $\sigma$-conics.

\begin{proposition} \label{pr2-sigma}
Let $X$ be a general ordinary GM fourfold and $C$ be a $\sigma$-conic on $X$. Then we have a triangle
\[\mathbb{D}(I_q(H))[1]\to \pr_2(I_C)\to \cQ^{\vee},\]
where $q$ is the unique $\sigma$-quadric on $X$.
\end{proposition}

\begin{proof}
Let $Y$ be a hyperplane section of $X$ containing $C$, it is easy to see $Y$ is integral because $X$ is of Picard number one.
Since $X$ is general, it does not contain any plane, then $\langle C \rangle \cap Y=C$. At the same time, $C$ is a $\sigma$-conic in $Y$, which means that $C$ is the zero locus of a section of $\cQ_Y$. Thus we have an exact sequence on $Y$
\[0\to \cU_Y\to \cQ^{\vee}_Y\to I_{C/Y}\to 0.\]
Note that we have exact sequences on $X$
\begin{equation}\label{U}
 0\to \cU(-H)\to \cU\to \cU_Y\to 0  
\end{equation}
and
\begin{equation}\label{Q dual}
0\to \cQ^{\vee}(-H)\to \cQ^{\vee}\to \cQ^{\vee}_Y\to 0.   
\end{equation}

Hence $\RHom(\cU^{\vee}, \cU_Y)=k[-3]$ and $\RHom(\cU^{\vee}, \cQ^{\vee}_Y)=k[-1]$. Applying the mutation~$\bL_{\oh_X}$ to the defining complexes of $\bL_{\cU^{\vee}}\cU_Y$ and $\bL_{\cU^{\vee}}\cQ^{\vee}_Y$ respectively,
we get triangles
\[\cU_Y\to \bL_{\oh_X}\bL_{\cU^{\vee}}\cU_Y\to \cQ^{\vee}[-1]\]
and
\[\cQ^{\vee}\to \cQ^{\vee}_Y\to \bL_{\oh_X}\bL_{\cU^{\vee}}\cQ^{\vee}_Y.\]
Now applying the mutation $\bR_{\oh_X(-H)}$ on the sequences (\ref{U}) and (\ref{Q dual}), since $\Ext^1(\cU, \cQ^{\vee})=0$ by Lemma~ \ref{coho-1},
we have $\bR_{\oh_X(-H)}\cU_Y\cong \cU\oplus \cQ(-H)$ and $\bR_{\oh_X(-H)}\cQ_Y^{\vee}\cong \cU\oplus \cQ^{\vee}$. 

Applying $\bR_{\oh_X(-H)}$ to the above triangles respectively, we get triangles
\[\cU\oplus \cQ(-H)\to \bR_{\oh_X(-H)}\bL_{\oh_X}\bL_{\cU^{\vee}}\cU_Y\to \cQ^{\vee}[-1]\]
and
\[\cQ^{\vee}\to \cU\oplus \cQ^{\vee}\to \bR_{\oh_X(-H)}\bL_{\oh_X}\bL_{\cU^{\vee}}\cQ^{\vee}_Y.\]

After taking the mutation $\bR_{\cU}$, we have 
\[\bR_{\cU}\cQ(-H)\to \pr_2(\cU_Y)\to \cQ^{\vee}[-1]\]
and
\[\pr_2(\cQ^{\vee}_Y)\cong0.\]
Therefore, using Lemma~\ref{pr2-3fold}, combined with the sequence $0\to \cU_Y\to \cQ^{\vee}_Y\to I_{C/Y}\to 0$ and $\pr_2(\cQ^{\vee}_Y)\cong0$, we obtain that 
\[\pr_2(I_{C})\cong\pr_2(\cU_Y)[1].\]
Under this case, $\pr_2(I_{C})$ sits in the triangle 
\[\bR_{\cU}\cQ(-H)[1]\to\pr_2(I_C)\to \cQ^{\vee}.\]
Now the result follows from $\bR_{\cU}\cQ(-H)=\mathbb{D}(I_q(H))$. To this end, we only need to prove that $I_q(H)\cong\mathbb{D}(\bR_{\cU}\cQ(-H))\cong\bL_{\cU^{\vee}}\cQ^{\vee}(H)$, which is implied by the fact $\bL_{\cU}\cQ^\vee\cong I_q$ in Proposition~\ref{projection_tauto}.
\end{proof}

\section{Stability of projection objects of conics}\label{stability_projection_conics}
In this section, we apply Proposition~\ref{stability_object_ext1_small} to show $\mathrm{pr}_2(I_C)$ is stable with respect to every stability condition on $\Ku(X)$ for a \emph{very general} ordinary GM fourfold~$X$. 

\begin{theorem} \label{pr_stable_thm}
Let $X$ be a very general ordinary GM fourfold and $C$ be a conic on $X$. Then the objects $\mathrm{pr}_2(I_C)$ and $\pr_1(\oh_C(H))$ are stable with respect to every stability condition on $\Ku(X)$.
\end{theorem}

\begin{proof}
By Proposition \ref{stability_object_ext1_small} and the identification in Lemma \ref{T-action}, for a conic $C$, we only need to show that one of the objects $\pr_1(\oh_C(H))$ and $\pr_2(I_C)$ is stable. In the followings, we prove the theorem in Proposition~ \ref{stability of tau_conic}, Proposition~\ref{stability_rho_conic} and Proposition~\ref{stability_projection_sigma_conic}.
\end{proof}

\subsection{Stability of projection objects of $\tau$-conics}
At first, we list two lemmas that are useful when we compute the spectral sequences in Proposition~\ref{stability of tau_conic}.
\begin{lemma}\label{ICIS}
Let $X$ be a general ordinary GM fourfold and $C$ be a $\tau$-conic on $X$. Let $\Sigma$ be the zero locus of a section of $\cU^{\vee}$ containing $C$. Then we have

\begin{enumerate}
    \item $\RHom(I_{C},I_{C})=k[0]\oplus k^{5}[-1]\oplus k^7[-2]$.
    
    \item $\RHom(I_{\Sigma}, I_{\Sigma})=k[0]\oplus k^4[-1]$
    
    \item $\RHom(I_{C}, I_{\Sigma})=k[-1]\oplus k^4[-2]$.
    
    \item $\RHom(I_{\Sigma}, I_{C})=k[0]\oplus k^6[-1]$.
\end{enumerate}
\end{lemma}

\begin{proof}
(1): It is clear that $\hom(I_{C},I_{C})=1$, by Serre duality we have $$\ext^4(I_{C},I_{C})=\hom(I_{C},I_{C}(-2H))=0.$$ 
Using \cite[Theorem 3.2]{iliev2011fano}, we know $\ext^1(I_{C}, I_{C})=5$. Since $\chi(I_{C},I_{C})=3$, we only need to show that $\ext^3(I_{C},I_{C})=0$. To this end, we apply $\Hom(I_{C},-)$ to the exact sequence 
$$0\to I_{C}(-2H)\to \oh_X(-2H)\to \oh_C(-2H)\to 0.$$ 
Since $\Hom(I_{C}, \oh_X(-2H))\cong\Ext^1(I_{C}, \oh_X(-2H))=0,$ we have $$\Ext^1(I_{C}, I_{C}(-2H))\cong\Hom(I_{C}, \oh_C(-2H))=0.$$ 
By Serre duality, we obtain $\Ext^3(I_{C},I_{C})\cong\Ext^1(I_{C}, I_{C}(-2H))=0$.

(2): Note that $\chi(I_{\Sigma}, I_{\Sigma})=-3$. Recall that $\Sigma$ is the zero locus of a section of $\cU^{\vee}$, hence we have the Koszul resolution
\[0\to \oh_X(-H) \to  \cU \to I_{\Sigma}\to 0.\]
Then the result follows from applying \cite[Lemma 2.27]{pirozhkov2020admissible} to this exact sequence.

(3): It is clear that $\Hom(I_{C}, I_{\Sigma})=\Ext^4(I_C, I_{\Sigma})=0$. 
Now the result follows from applying $\Hom(I_{C},-)$ to the Koszul resolution of $I_{\Sigma}$.

(4): Applying $\Hom(-, I_{C})$ to the resolution of $I_{\Sigma}$, we obtain $\Ext^i(I_{\Sigma}, I_{C})=0$ for $i\neq 0,1$ and an exact sequence
\[0\to \Hom(I_{\Sigma}, I_{C})\to k\to k^6\to \Ext^1(I_{\Sigma}, I_{C})\to 0. \]
Since $\hom(I_{\Sigma}, I_{C})= 1$, we obtain $\ext^1(I_{\Sigma}, I_{C})=6$.
\end{proof}

\begin{lemma} \label{ICSexts1}
Let $X$ be a general ordinary GM fourfold and $C$ be a  $\tau$-conic on $X$. Let $\Sigma$ be the zero locus of a section of $\cU^{\vee}$ containing $C$. Then we have:

\begin{enumerate}
    \item $\RHom(I_{C/\Sigma}, I_{C/\Sigma})=k[0]\oplus k^n[-1]\oplus k^{n+1}[-2]$ for some $n\geq 3$.
    
    \item $\RHom(I_{C/\Sigma}(H), \oh_X)=k^2[-2]$.
    
    \item $\RHom(\oh_X, I_{C/\Sigma}(H))=k^2[0]$.
\end{enumerate}
\end{lemma}

\begin{proof}
Note that $\chi(I_{C/\Sigma}, I_{C/\Sigma})=2$. Then (1) follows from Lemma \ref{ICIS} and applying \cite[Lemma 2.27]{pirozhkov2020admissible} to the exact sequence $0\to I_{\Sigma}\to I_{C}\to I_{C/\Sigma}\to 0$.

(2) and (3) follow from applying $\Hom(-, \oh_X)$ and $\Hom(\oh_X, -)$ to the exact sequence 
$0\to I_\Sigma(H)\to I_C(H)\to I_{C/\Sigma}(H)\to 0.$ 
\end{proof}

\begin{proposition}
\label{stability of tau_conic}
Let $X$ be a very general ordinary GM fourfold and $C$ be a $\tau$-conic on $X$. Then  $\pr_1(\oh_C(H))$ is stable with respect to every stability condition on $\Ku(X)$. 
\end{proposition}

\begin{proof}
In Proposition~\ref{proj-tau}, we have $\mathrm{pr}_1(\oh_C(H))\cong\bL_{\oh_X}(I_{C/\Sigma}(H))[1]$. Then we apply \cite[Lemma 2.27]{pirozhkov2020admissible}
to the triangle 
$$\oh_X^{\oplus 2}\to I_{C/\Sigma}(H)\to \pr_1(\oh_C(H))[-1].$$ 
From Lemma \ref{ICSexts1}, for $i\notin \{0,1,2\}$, $$\Hom(\pr_1(\oh_C(H)), \pr_1(\oh_{C}(H)))=k, \quad\Ext^i(\pr_1(\oh_C(H)), \pr_1(\oh_{C}(H)))=0.$$ 
Then by Serre duality in $\Ku(X)$, we have $$\Ext^2(\pr_1(\oh_C(H)), \pr_1(\oh_{C}(H)))=\Hom(\pr_1(\oh_C(H)), \pr_1(\oh_{C}(H)))=k.$$
Since $\chi(\pr_1(\oh_C(H)), \pr_1(\oh_{C}(H)))=-2$, we obtain $\Ext^1(\pr_1(\oh_C(H),\pr_1(\oh_C(H))=k^4$. Then by Proposition~\ref{stability_object_ext1_small}, $\pr_1(\oh_C(H))$ is stable with respect to every stability condition on $\Ku(X)$. 
\end{proof}

\subsection{Stability of projection objects of $\rho$-conics}
Let $C$ be a $\rho$-conic on $X$, by Proposition~\ref{pr2-rho}, $\mathrm{pr}_2(I_C)\cong\mathrm{pr}_1(\cU)[1]$, where $\mathrm{pr}_1(\cU)$ fits into the triangle as in Proposition~\ref{projection_tauto}
$$\cU\rightarrow\mathrm{pr}_1(\cU)\rightarrow\bL_{\oh_X}(I_q(H))[-2].$$ 
Now we only need to prove $\mathrm{pr}_1(\cU)$ is stable.
\begin{lemma}\label{ext_group_pr_U}\leavevmode
\begin{enumerate}
    \item $\RHom(\cU,\bL_{\oh_X}(I_q(H))[-3])=k^m[-2]\oplus k^{m-3}[-3]$, for some integer $3\leq m\leq 25$.
    \item $\RHom(\cU,\mathrm{pr}_1(\cU))\cong\RHom(\pr_1(\cU),\mathrm{pr}_1(\cU))=k[0]\oplus k^4[-1]\oplus k[-2]$.
\end{enumerate}
\end{lemma}

\begin{proof}
(1): Applying $\Hom(\cU,-)$ to the tautological exact sequence and using (1) of Lemma \ref{coho-1}, we see $\RHom(\cU, \cU^{\vee})=k^{24}[0]$. And we have $\RHom(\cU, \cQ^{\vee}(H))=k^{46}[0]$ by (6) of Lemma \ref{coho-1}. Therefore, applying $\mathrm{Hom}(\cU,-)$ to the short exact sequence
    $$0\rightarrow\cU^\vee\rightarrow\cQ^{\vee}(H)\rightarrow I_q(H)\rightarrow 0,$$
we get $\RHom(\cU,I_q(H))=k^{22}[0]$. Then we apply $\mathrm{Hom}(\cU,-)$ to the triangle
$$\oh_X^{\oplus5}\rightarrow I_q(H)\rightarrow\bL_{\oh_X}(I_q(H)).$$
Since $\RHom(\cU,\oh_X^{\oplus5})=k^{25}[0]$, by the long exact sequence we have $$\RHom(\cU,\bL_{\oh_X}(I_q(H)))=k^m[1]\oplus k^{m-3}[0],$$
for some integer $3\leq m\leq 25$. Then we obtain $$\RHom(\cU,\bL_{\oh_X}(I_q(H))[-3])=k^m[-2]\oplus k^{m-3}[-3].$$

(2): Since $\mathrm{pr}_1(\cU)$ fits into the triangle
$$\bL_{\oh_X}(I_q\otimes\oh_X(H))[-3]\rightarrow\cU\rightarrow\mathrm{pr}_1(\cU),$$
using (1) and $\RHom(\cU,\cU)=k[0]$, 
we have $\RHom(\cU,\mathrm{pr}_1(\cU))=k[0]\oplus k^m[-1]\oplus k^{m-3}[-2]$. By Serre duality in $\Ku(X)$ and adjunction, we have
\[\Hom(\cU, \pr_1(\cU))=\Hom(\pr_1(\cU), \pr_1(\cU))=\Hom(\pr_1(\cU), \pr_1(\cU)[2])=\Hom(\cU, \pr_1(\cU)[2]).\]
Therefore, we obtain $m-3=1$, which means $\RHom(\cU,\mathrm{pr}_1(\cU))=k[0]\oplus k^4[-1]\oplus k[-2]$. 
\end{proof}

\begin{proposition}
\label{stability_rho_conic}
Let $X$ be a very general ordinary GM fourfold and $C$ be a $\rho$-conic on~$X$, then $\mathrm{pr}_2(I_C)$ is stable with respect to every stability condition on $\Ku(X)$. 
\end{proposition}

\begin{proof}
By Lemma~\ref{ext_group_pr_U}, $\mathrm{Ext}^1(\mathrm{pr}_1(\cU),\mathrm{pr}_1(\cU))=k^4$. Then using Proposition~\ref{stability_object_ext1_small}, $\mathrm{pr}_1(\cU)$ is stable, which implies $\mathrm{pr}_2(I_C)$ is stable with respect to every stability condition on $\Ku(X)$. 
\end{proof}

\subsection{Stability of projection objects of $\sigma$-conics}
\begin{proposition}
\label{stability_projection_sigma_conic}
Let $X$ be a very general GM fourfold and $C$ be a $\sigma$-conic on $X$, then $\mathrm{pr}_2(I_C)$ is stable with respect to every stability condition on $\Ku(X)$. 
\end{proposition}

\begin{proof}
By Proposition~\ref{pr2-sigma}, the object $\mathrm{pr}_2(I_C)$ fits into the triangle
$$\mathbb{D}(I_q(H))[1]\rightarrow\mathrm{pr}_2(I_C)\rightarrow\cQ^{\vee}.$$
It is easy to check that $T(\mathrm{pr}_2(I_C))\cong\mathrm{pr}_1(\cU)[1]$. Then the result follows from Proposition \ref{stability_rho_conic}.
\end{proof}

\begin{remark}
\label{involution_switch_rho_and_sigma_conic}
For any $\rho$-conic $C$, $\mathrm{pr}_2(I_C)\cong\mathrm{pr}_1(\cU)[1]$, for any $\sigma$-conic $C$, we have  $T(\mathrm{pr}_2(I_C))\cong\mathrm{pr}_1(\cU)[1]$. Indeed, it is easy to check $T(\mathrm{pr}_1(\cU)[1])\ncong\mathrm{pr}_1(\cU)[1]$. This means that the morphism $p$ induced by $\mathrm{pr}_2$ contracts the locus $\sigma$-conics and $\rho$-conics in $F_g(X)$ to two different points in the moduli space $\mathcal{M}_{\sigma}(\Ku(X),\Lambda_1)$. Moreover, the induced action of $T$ on $\mathcal{M}_{\sigma}(\Ku(X),\Lambda_1)$ will take one point to another. 
\end{remark}

\section{Bridgeland moduli spaces and the double EPW sextics}\label{moduli_spaces_construction}
\subsection{Moduli space of stable objects $\mathcal{M}_{\sigma}(\Ku(X),\Lambda_1)$}
For a very general GM fourfold~$X$, let $F_g(X)$ be the Hilbert scheme of conics on $X$.
In this section, we show that the projection functor $\mathrm{pr}_1:\D^b(X)\rightarrow\Ku(X)$(or equivalently, $\mathrm{pr}_2$) induces a dominant proper morphism~$p$ from $F_g(X)$ to the Bridgeland moduli space $\mathcal{M}_{\sigma}(\Ku(X),\Lambda_1)$. Moreover, this morphism~$p$ is compatible with the morphism $f:F_g(X)\rightarrow\widetilde{Y}_{A^{\perp}}$ defined in \cite[Section 4.4]{iliev2011fano}. In particular, we show that  $\mathcal{M}_{\sigma}(\Ku(X),\Lambda_1)\cong\widetilde{Y}_{A^{\perp}}$ and $\mathcal{M}_{\sigma}(\Ku(X),\Lambda_2)\cong\widetilde{Y}_{A}$.

By Theorem \ref{pr_stable_thm}, the object $\mathrm{pr}_1(\oh_C(H))$ is $\sigma$-stable for any conic $C\in F_g(X)$. On the other hand, $F_g(X)$ admits a universal family and the functor $\mathrm{pr}_1$ is of Fourier--Mukai type. By the standard argument as in \cite[Theorem 3.9]{li2018twisted}, the projection functor $\mathrm{pr}_1$ induces a morphism $p:F_g(X)\rightarrow\mathcal{M}_{\sigma}(\Ku(X),\Lambda_1)$. According to Remark~\ref{involution_switch_rho_and_sigma_conic}, the morphism $p$ contracts $\rho$-conics and $\sigma$-conics to two different points in $\mathcal{M}_{\sigma}(\Ku(X),\Lambda_1)$. Next, when restricting on the locus of $\tau$-conics, we show that $p$ is a $\mathbb{P}^1$-bundle.

We first review the construction of $f: F_g(X)\rightarrow \widetilde{Y}_{A^\perp}$. In \cite[Section 4.4]{iliev2011fano}, $f(C)=f(C')$ if and only if

\begin{enumerate}
    \item there exists some  
$V_4\subset V_5$ such that $C,C'\subset \mathrm{Gr}(2,V_4)\cap X$, 

\item the planes $\langle C \rangle$ and $\langle C' \rangle$ in $ \mathbb{P}(\bigwedge^2 V_4)\cap H\cong \mathbb{P}^4$ are contained in a same quadric $Q'$, where $Q'$ is in the pencil  $|P_{V_4},Q_{V_4}|$, where $Q_{V_4}:=Q\cap \mathbb{P}(\bigwedge^2 V_4)\cap H$ and $P_{V_4}:=\mathrm{Gr}(2,V_4)\cap H$,

\item $\langle C \rangle$ and $\langle C' \rangle$ are linearly equivalent divisors on $Q'$.
\end{enumerate}



    
    


\begin{proposition}\label{contracts_tau_conic_by_P1}
Let $C$ and $C'$ be two $\tau$-conics on $X$. Then $f(C)=f(C')$ if and only if $\Hom(\pr_1(\oh_{C}(H)), \pr_1(\oh_{C'}(H)))\neq 0$.
\end{proposition}

\begin{proof}

First we assume that $C\subset \Sigma$ and $C'\subset \Sigma'$, where the two surfaces $\Sigma:=\mathrm{Gr}(2,V_4)\cap X$ and $\Sigma':=\mathrm{Gr}(2,V_4')\cap X$ are zero locus of two sections of $\cU^\vee$. 

From the definition of left mutation and Proposition~\ref{proj-tau}, we have two triangles $$\oh_X^{\oplus 2}\to I_{C/\Sigma}(H)\to \pr_1(\oh_C(H))[-1]$$ and $$\oh_X^{\oplus 2}\to I_{C'/\Sigma'}(H)\to \pr_1(\oh_{C'}(H))[-1].$$
If we apply \cite[Lemma 2.27]{pirozhkov2020admissible} to these two triangles, we obtain a spectral sequence with the first page $E^{p,q}_1$ in the form

$$       \begin{array}{c|cc}
         \vdots & \vdots & \vdots  \\
	     \Ext^1(I_{C/\Sigma}(H), \oh_X^{\oplus 2}) & \Ext^1(\oh_X^{\oplus 2}, \oh_X^{\oplus 2})\oplus \Ext^1(I_{C/\Sigma}(H), I_{C'/\Sigma'}(H)) & \Ext^1(\oh_X^{\oplus 2},I_{C'/\Sigma'}(H)) \\
	     \Hom(I_{C/\Sigma}(H), \oh_X^{\oplus 2}) & \Hom(\oh_X^{\oplus 2}, \oh_X^{\oplus 2})\oplus \Hom(I_{C/\Sigma}(H), I_{C'/\Sigma'}(H)) & \Hom(\oh_X^{\oplus 2},I_{C'/\Sigma'}(H)) \\ \hline
	     0 &  0 & 0
	\end{array}$$
Since $\Hom(I_{C/\Sigma}(H), \oh_X^{\oplus 2})=\Ext^1(I_{C/\Sigma}(H), \oh_X^{\oplus 2})=0$ and $\pr_1(\oh_{C'}(H))\in \Ku(X)$,
\[\Hom(\pr_1(\oh_C(H)), \pr_1(\oh_{C'}(H)))=E^{0,0}_{\infty}=\ker(E^{0,0}_1\to E^{1,0}_1)=\Hom(I_{C/\Sigma}, I_{C'/\Sigma'}).\] 
Then we need to prove that $\Hom(I_{C/\Sigma}, I_{C'/\Sigma'})\neq 0$ if and only if $f(C)=f(C')$.


First we claim that if  $\Hom(I_{C/\Sigma}, I_{C'/\Sigma'})\neq 0$, then $\Sigma=\Sigma'$, i.e., $\Hom(I_{\Sigma/\mathbb{P}^4}, I_{\Sigma'/\mathbb{P}^4})\neq 0$. Indeed, this follows from applying \cite[Lemma 2.27]{pirozhkov2020admissible} to the exact sequences $$0\to I_{\Sigma/\mathbb{P}^4}\to I_{C/\mathbb{P}^4}\to I_{C/\Sigma}\to 0$$and $$0\to I_{\Sigma'/\mathbb{P}^4}\to I_{C'/\mathbb{P}^4}\to I_{C'/\Sigma'}\to 0.$$

Assume that $f(C)=f(C')$. Then we have $\Sigma=\Sigma'$, and the planes $\langle C \rangle$ and $\langle C' \rangle$ in are contained in a same quadric $Q'$, where $Q'$ is in the pencil  $|P_{V_4},Q_{V_4}|$. As in \cite[Proposition 4.9]{iliev2011fano}, $Q'$ is either a cone  over a smooth quadric surface and planes in $Q'$ are parametrized by two projective lines, or a double cone over a smooth conic and planes in $Q'$ are parametrized by that smooth conic. Thus, there is a one-to-one correspondence between a $\tau$-conic $C$ with $\langle C\rangle \subset Q'$ and a plane contained in $Q'$.
Since $f(C)=f(C')$, from the construction of $f$, we know that $\langle C \rangle$ and $\langle C' \rangle$ are linearly equivalent and hence come from the same family of planes in $Q'$. This means that they are linearly equivalent as Weil divisors in $\Sigma$, which implies $I_{C/\Sigma}\cong I_{C'/\Sigma}$. Hence we have $\Hom(I_{C/\Sigma}, I_{C'/\Sigma})\neq 0$.

Conversely, if $\Hom(I_{C/\Sigma}, I_{C'/\Sigma'})\neq 0$, then we know that $\Sigma=\Sigma'$ as we claimed above. Moreover, $C$ and $C'$ are linearly equivalent as Weil divisors in $\Sigma$. From the construction of $f$, we obtain that $f(C)=f(C')$. 
\end{proof}

Now, we are ready to prove the first main result of our paper.

\begin{theorem}\label{first_moduli_space}
Let $X$ be a very general GM fourfold with a Lagrangian data~$(A,V_5,V_6)$. For any generic stability condition $\sigma$ on $\Ku(X)$, the projection functor $\pr_1$ will induce an isomorphism
$$i:\widetilde{Y}_{A^{\bot}}\cong\mathcal{M}_\sigma(\Ku(X),\Lambda_1).$$
\end{theorem}

\begin{proof}
Now we have the morphism  $p:F_g(X)\rightarrow\mathcal{M}_{\sigma}(\Ku(X),\Lambda_1)$ induced by $\pr_1$.
In Proposition~\ref{contracts_tau_conic_by_P1}, on the locus of $\tau$-conics, $p$ is a $\mathbb{P}^1$-bundle and coincides with $f$. Furthermore, it follows from Proposition~\ref{pr2-rho} and Proposition~\ref{pr2-sigma} that the morphism $p$ coincides with $f$ on the whole $F_g(X)$. By \cite[Proposition 1.5]{perry2019stability}, we know that $\mathcal{M}_\sigma(\Ku(X),\Lambda_1)$ is a smooth projective  variety of dimension four. On the other hand, since the dimension of $F_g(X)$ is five and the general fiber of $p$ is one dimensional, $p$ is a proper dominant morphism onto $\mathcal{M}_\sigma(\Ku(X),\Lambda_1)$. Then according to  \cite[Exercise 29.5.C]{vakilFOAG}, we have an isomorphism $i:\widetilde{Y}_{A^{\bot}}\cong\mathcal{M}_\sigma(\Ku(X),\Lambda_1)$ such that $i\circ f=p$.
\end{proof}

\subsection{Moduli space of stable objects $\mathcal{M}_{\sigma}(\Ku(X),\Lambda_2)$}
In Theorem~\ref{first_moduli_space}, we have shown that for a very general GM fourfold $X$, the moduli space $M_\sigma(\Ku(X),\Lambda_1)$ is isomorphic to the double dual EPW sextic  $\widetilde{Y}_{A^\perp}$. Now we prove the moduli space $\mathcal{M}_\sigma(\Ku(X),\Lambda_2)$ is isomorphic to another hyperkähler fourfold, the double EPW sextic $\widetilde{Y}_{A}$. 

\begin{theorem}
\label{second_moduli_space}
Let $X$ be a very general GM fourfold with a Lagrangian data~$(A,V_5,V_6)$. For any generic stability condition $\sigma$ on $\Ku(X)$,  we have an isomorphism  $$\mathcal{M}_{\sigma}(\Ku(X),\Lambda_2)\cong\widetilde{Y}_{A}.$$ 
\end{theorem}

\begin{proof}
Let $X'$ be the period dual of $X$, we can always choose $X'$ to be very general as well. Then by \cite[Theorem 1.6]{kuznetsov2019categorical}, there is an equivalence of Fourier--Mukai type $\Phi:\Ku(X)\simeq\Ku(X')$. We claim that $\Phi$ induces an isomorphism between Bridgeland moduli spaces $\phi:\mathcal{M}_{\sigma}(\Ku(X),\Lambda_1)\rightarrow\mathcal{M}_{\sigma'}(\Ku(X'),\Lambda_2')$. Indeed as a corollary of \cite[Theorem 5.12]{bayer2022kuznetsov}, the induced isomorphism $$[\Phi]: \mathcal{N}(\Ku(X))\cong \mathcal{N}(\Ku(X'))$$ will identify the canonical rank 2 lattices $\langle\Lambda_1,\Lambda_2\rangle $ and $\langle\Lambda_1',\Lambda_2'\rangle$ on each side. Then up to sign, $[\Phi](\Lambda_1)=\Lambda_1'$ or $\Lambda_2'$. If $[\Phi](\Lambda_1)=\Lambda_1'$, then $\Phi$ induces a bijective map $\phi$ between the moduli spaces $\mathcal{M}_{\sigma}(\Ku(X),\Lambda_1)$ and $\mathcal{M}_{\sigma'}(\Ku(X),\Lambda_1')$. Since $\Phi$ is of Fourier--Mukai type, we get an isomorphism of moduli functors $\mathsf{M}_1\cong \mathsf{M}'_1\colon$ $(\mathrm{Sch}/\mathbb{C})^{op}\to \mathrm{Gpds}$, where 
\[\mathsf{M}_i\colon T\mapsto \{\cF\in \D^b(X\times T)~ \text{a family of geometrically}~\sigma\text{-stable objects over}~ T ~\text{of class} ~\Lambda_i\}\]
\[\mathsf{M}'_i\colon T\mapsto \{\cF'\in \D^b(X'\times T)~ \text{a family of geometrically}~\sigma'\text{-stable objects over}~ T ~\text{of class} ~\Lambda'_i\},\]
see \cite[Definition 21.11]{BLMNPS21}.
More precisely, the isomorphism of functors is given by
\[\mathsf{M}_1(T)\to \mathsf{M}_1'(T),\quad\cF\mapsto \overline{\Phi}\boxtimes \mathrm{id}_{\D^b(T)}(\cF)\]
for any $T\in (\mathrm{Sch}/\mathbb{C})^{op}$, where $\overline{\Phi}$ is the composition $\D^b(X)\xra{\pr_1}\Ku(X)\xra{\Phi}\Ku(X')\hookrightarrow \D^b(X')$. Note that the functor $\overline{\Phi}\boxtimes \mathrm{id}_{\D^b(T)}$ makes sense since $\overline{\Phi}$ and $\mathrm{id}_{\D^b(T)}$ are Fourier--Mukai functors so that we can pull their Fourier--Mukai kernels back to $X\times T\times X'\times T$ and take the product.

Since both moduli functors are algebraic stacks finite type over $\mathbb{C}$ (cf.~\cite[Theorem 21.24 (2)]{BLMNPS21}), by the uniqueness of good moduli spaces \cite[Theorem 6.6]{alp} we have 
$$\phi:\widetilde{Y}_{A^{\perp}}\cong\mathcal{M}_{\sigma}(\Ku(X), \Lambda_1)\cong\mathcal{M}_{\sigma'}(\Ku(X'),\Lambda_1')\cong\widetilde{Y}_{A'^{\perp}}.$$

But $X'$ is the period dual of $X$, $A'=A^{\perp}$, then we get an isomorphism $\widetilde{Y}_{A^{\perp}}\cong\widetilde{Y}_{A}$, which is impossible by \cite[Theorem 1.1]{o2006dual}. Thus up to sign, we have $[\Phi](\Lambda_1)=\Lambda_2'$ and $[\Phi](\Lambda_2)=\Lambda_1'$, which implies  $\mathsf{M}_1\cong \mathsf{M}'_2$ by the same argument above. Then the moduli space $\mathcal{M}_{\sigma}(\Ku(X),\Lambda_2)\cong\mathcal{M}_{\sigma}(\Ku(X'),\Lambda_1')\cong\widetilde{Y}_{A'^{\perp}}$, again by $A'=A^{\perp}$, we have  $\mathcal{M}_{\sigma}(\Ku(X),\Lambda_2)\cong\widetilde{Y}_{A}$.
\end{proof}

\begin{remark}
In fact, \cite[Theorem 5.12]{bayer2022kuznetsov} holds for any GM fourfold. As a result, once we extend Theorem~\ref{main_theorem_1}(1) to a general GM fourfold, Theorem~\ref{main_theorem_1}(2) automatically holds. The point is that even though the rank of the numerical Grothendieck group is bigger than two, the induced equivalence of period duals still fixes the canonical $A_1^{\oplus2}$ lattice. Then our method in Theorem~\ref{second_moduli_space} still works.
\end{remark}

\subsection{An involution on $\Ku(X)$ and its induced action on the double dual EPW sextic}
Now we are going to discuss two involutions acting on $\widetilde{Y}_{A^{\perp}}$ and $\mathcal{M}_{\sigma}(\Ku(X),\Lambda_1)$ respectively. One is naturally induced by the structure of double cover and the other is induced by the involutive functor $T$ of $\Ku(X)$ defined in Lemma \ref{involution_functor_original}.

By the result of Theorem~\ref{first_moduli_space}, we have the following diagram
\[
\begin{tikzcd}
            & \widetilde{Y}_{A^\perp} \arrow[d, "i"] \arrow[r, "\eta"] & \widetilde{Y}_{A^\perp} \\
F_g(X) \arrow[r, "p"] \arrow[ru, "f"] & \mathcal{M}_{\sigma}(\Ku(X),\Lambda_1) \arrow[r, "\eta'"]                & \mathcal{M}_{\sigma}(\Ku(X),\Lambda_1),
\end{tikzcd}
\]
where $i$ is the isomorphism such that $i\circ f=p$ in Theorem~\ref{first_moduli_space} and the two involutions are denoted by $\eta$ and $\eta'$ respectively. To prove $\eta$ and $\eta'$ coincide, it suffices to show that for a general conic $C\in F_g(X)$, there exists another conic $C'$ such that $\eta\circ f(C)=f(C')$ and $\eta'\circ p(C)=p(C')$. 

Firstly, we briefly review the involution $\eta$ described in \cite[Lemma 4.19]{iliev2011fano}. For a general conic $C$, if there exists another conic $C'$ such that $\eta\circ f(C)=f(C')$,  then the spanning planes $\langle C \rangle$ and $\langle C' \rangle$ lie in a same quadric threefold $Q_{C,V_4}=Q_{C',V_4}$. Here $Q_{C,V_4}$ is the unique singular quadric of the pencil $\langle\mathrm{Gr}(2,V_4)\cap H, H\cap Q\rangle$ contained in $\mathbb{P}(\wedge^2V_4)\cap H$. However, $\langle C \rangle$ and $\langle C' \rangle$ do not belong to the same ruling, as a result, $\langle C \rangle$ meets $\langle C' \rangle$ along a line. Furthermore, the generating 3-plane  $\langle C,C' \rangle$ cuts $X$ at a degenerate elliptic curve 
\[e:=C\cup C'=\langle C,C' \rangle\cap\mathrm{Gr}(2,V_4)\cap Q.\]

Now we only need to show such $C'$ satisfies $\eta'\circ p(C)=p(C')$ as well, which is equivalent to prove $T(\mathrm{pr}_1(\oh_{C}(H)))\cong\mathrm{pr}_1(\oh_{C'}(H))$.

\begin{lemma} \label{T-act-tau}
Let $X$ be a general ordinary GM fourfold.
Let $C$ be a general $\tau$-conic such that it is contained in a smooth surface $\Sigma$, defined by the zero locus of a section of $\cU^{\vee}$. Then we have a triangle
\[0\to I_{\Sigma}\to \pr_2(I_C)\to I_C\to 0.\]
Moreover, $T(\mathrm{pr}_1(\oh_{C}(H)))\cong\mathrm{pr}_1(\oh_{C'}(H))$ for another $\tau$-conic $C'\subset \Sigma$ such that as divisors of $\Sigma$, $C'=-C-K_{\Sigma}$.
\end{lemma}

\begin{proof}
The first statement follows from Proposition \ref{proj-tau}, note that $\mathrm{cok}(\cU^{\oplus 4}\hookrightarrow K_1)=\mathrm{cok}(\oh_X(-H)\hookrightarrow \cU)\cong I_{\Sigma}$. Since $I_{\Sigma}$ and $I_C$ are both subsheaves of $\oh_X$, $\mathrm{pr}_2(I_C)$ is also a subsheaf of $\oh_X^{\oplus 2}$. The cokernel of $\mathrm{pr}_2(I_C)\xhookrightarrow{i}\oh_X^{\oplus 2}$ is a sheaf obtained by an extension of $\oh_{\Sigma}$ by $\oh_C$. Consider the standard exact sequence of $C\subset\Sigma$,
$$0\rightarrow \oh_{\Sigma}(-C)\rightarrow\oh_{\Sigma}\rightarrow\oh_C\rightarrow 0.$$

By the standard exact sequence $0\to \cU^{\vee}|_C=N_{\Sigma/X}|_C\to N_{C/X}\to N_{C/\Sigma}\to 0$ and \cite[Lemma 3.3]{iliev2011fano}, we see $\deg N_{C/\Sigma}=0$, i.e.~$C.C=0$. Then  tensoring with $\oh_{\Sigma}(C)$, we get 
$$0\rightarrow\oh_{\Sigma}\rightarrow\oh_{\Sigma}(C)\rightarrow\oh_C\rightarrow 0.$$ 
Then $\mathrm{pr}_2(I_C)\cong\ker(\oh_X^{\oplus 2}\rightarrow\oh_\Sigma(C))$.
Note that $N_{\Sigma/X}=\cU^{\vee}|_C$, then from the conormal sequence and $K_X=\oh_X(-2H)$, we see  $K_{\Sigma}\cong\oh_{\Sigma}(-H)$. Hence $-C-K_{\Sigma}$ is a divisor of conic as well. 

We take $C'\in |-C-K_{\Sigma}|$, then $C\cup C'\in |C'+C|=|-K_{\Sigma}|$, which is a degenerate degree $4$ elliptic curve on $\Sigma$.
Thus we have 
$$\mathrm{pr}_2(I_C)\cong\ker(\oh_X^{\oplus 2}\rightarrow\oh_\Sigma(-C'+H))\cong\ker(\oh_X^{\oplus 2}\rightarrow I_{C'/\Sigma}(H))\cong\pr_1(\oh_{C'}(H))[-2].$$
By Lemma \ref{T-action}, $\mathrm{pr}_2(I_C)\cong T(\mathrm{pr}_1(\oh_{C}(H)))[-2]$, we obtain $T(\mathrm{pr}_1(\oh_{C}(H)))\cong\mathrm{pr}_1(\oh_{C'}(H))$.

Now it remains to show that $C'$ is also a $\tau$-conic. By Proposition~\ref{pr2-rho} and Proposition~\ref{pr2-sigma}, if $C'$ is of $\sigma$-type or $\rho$-type, $\mathrm{pr}_1(\oh_{C'}(H))\cong\mathrm{pr}_1(\cU)[3]$ or $T(\mathrm{pr}_1(\cU))[3]$, neither of them is a rank two sheaf up to a  shift. Then the result follows. 
\end{proof}

\begin{proposition}\label{involution_induced_auto}
Let $X$ be a very general ordinary GM fourfold, then via the isomorphism~$i$, the two involutions $\eta$ and $\eta'$ coincide.
\end{proposition}

\begin{proof}
By Lemma \ref{T-act-tau}, for a general conic $C$, if $\eta'\circ p(C)=p(C')$ for another conic $C'$, $C$ and $C'$ lie in a same surface $\Sigma$ of degree four. In particular, $C\cup C'$ is an elliptic curve of degree 4, which spans a 3-plane in $\mathbb{P}(\wedge^2V_4)\cap H$, cutting along $X$ by~$C\cup C'$. This coincides with the choice of $C'$ in \cite{iliev2011fano}, which means for a general conic~$C$, we can always find another $C'$ such that $\eta\circ f(C)=f(C')$ and $\eta'\circ p(C)=p(C')$. In conclusion, the two involutions essentially are the same.

\end{proof}

\begin{remark}
There is an easier proof of Proposition~\ref{involution_induced_auto} using the fact that for a very general GM fourfold $X$, the automorphism group of $\widetilde{Y}_{A^\perp}$ is isomorphic to $\mathbb{Z}_2$, generated by the natural involution $\eta$. Thus to show $\eta$ coincides with $\eta'$, it suffices to prove $\eta'$ is non-trivial, which is obvious. However, our method in Proposition~\ref{involution_induced_auto} is independent of this fact. Thus Proposition~\ref{involution_induced_auto} will still be true in general once we identify $\widetilde{Y}_{A^{\perp}}$ with $\mathcal{M}_{\sigma}(\Ku(X), \Lambda_1)$ for a general ordinary GM fourfold $X$.
\end{remark}

\section{Birational categorical Torelli for Gushel--Mukai fourfolds}\label{KP_conjecture_GM}

\subsection{A universal family}

Let $X$ be a very general ordinary GM fourfold. Let $F^{\tau}_g(X)$ be the locus of $\tau$-conics and $M^X_0:=p(F^{\tau}_g(X))\subset \cM_\sigma(\mathcal{K}u(X), \Lambda_1)$. In this subsection, we construct a universal family on $M^X_0$. By \cite[Theorem 5.8]{kuz11}, we have a semiorthogonal decomposition of the form
\[\D^b(X\times F_g(X))=\langle \oh_X(-H)\boxtimes \D^b(F_g(X)), \cU\boxtimes \D^b(F_g(X)), \Ku(X\times F_g(X)), \]
\[\oh_X\boxtimes \D^b(F_g(X)), \cU^{\vee}\boxtimes \D^b(F_g(X)) \rangle.\]
Let $\pr_3:=\bR_{\cU\boxtimes \D^b(F_g(X))}\bR_{\oh_X(-H)\boxtimes \D^b(F_g(X))}\bL_{\oh_X\boxtimes \D^b(F_g(X))}\bL_{\cU^{\vee}\boxtimes \D^b(F_g(X))}$ be the relative projection functor. Let $\cI\in \Coh(X\times F_g(X))$ be the universal ideal sheaf of conics. We define \[\tilde{\cI}:=(\mathrm{id}\times p)_*\pr_3(\cI)\in \D^b(X\times \cM_\sigma(\mathcal{K}u(X), \Lambda_1))\]
and $\tilde{\cJ}$ be the restriction of $\tilde{\cI}$ to $X\times M^X_0$.

\begin{lemma} \label{uni_fam}
$\tilde{\cJ}$ is a universal family on $M^X_0$.
\end{lemma}

\begin{proof}
Let $[C]\in F_g^{\tau}(X)$ and we denote $[A_C]:=[\pr_2(I_C)]\in M^X_0\subset \cM_\sigma(\mathcal{K}u(X), \Lambda_1)$. We have a commutative diagram with all squares cartesian:
\[\begin{tikzcd}
	{X\times \mathbb{P}^1} & {X\times F^{\tau}_g(X)} & {X\times F_g(X)} \\
	{X\times \{[A_C]\}} & {X\times M^X_0} & {X\times \cM_\sigma(\mathcal{K}u(X), \Lambda_1)}
	\arrow["{q_1}", from=1-1, to=2-1]
	\arrow["{i_{[A_C]}}", hook, from=1-1, to=1-2]
	\arrow["{j_{[A_C]}}"', hook, from=2-1, to=2-2]
	\arrow["{\mathrm{id}\times p'}", from=1-2, to=2-2]
	\arrow["{l_2}", hook, from=1-2, to=1-3]
	\arrow["{l_1}"', hook, from=2-2, to=2-3]
	\arrow["{\mathrm{id}\times p}", from=1-3, to=2-3]
\end{tikzcd}\]
where $q_1$ is the projection to the first component, and rows are natural embeddings. Note that $p'$ is the restriction of $p$ to $F^{\tau}_g(X)$, which is a $\mathbb{P}^1$-fibration.
 
Now we have
\[j^*_{[A_C]}\tilde{\cJ}=j^*_{[A_C]}l_1^*\tilde{\cI}=j^*_{[A_C]}l_1^*(\mathrm{id}\times p)_*\pr_3(\cI)\]
\[\cong j^*_{[A_C]}(\mathrm{id}\times p')_*l^*_2\pr_3(\cI)\cong (q_1)_*i^*_{[A_C]}l_2^*\pr_3(\cI),\]
where the last two isomorphisms follow from the  base change theorem. For any $[C']\in \mathbb{P}^1\subset F_g(X)$, we denote $l_{[C']}:X\times \{[C']\}\hookrightarrow X\times \mathbb{P}^1$.
Note that since $\pr_3$ commute with base change, we have $l_{[C']}^*i^*_{[A_C]}l_2^*\pr_3(\cI)\cong \pr_2(I_{C'})= A_{[C']}\cong A_{[C]}$
for any $[C']\in \mathbb{P}^1$. Then from the construction of relative projection functor $\pr_3$, we have $i^*_{[A_C]}l_2^*\pr_3(\cI)\cong q_1^*A_C$.
Thus we obtain
\[j^*_{[A_C]}\tilde{\cJ}\cong (q_1)_*i^*_{[A_C]}l_2^*\pr_3(\cI)\cong  (q_1)_*q_1^*A_C\cong A_C,\]
where the last isomorphism follows from the projection formula and  $(q_1)_*\oh_{X\times \mathbb{P}^1}\cong \oh_X$. Since the above argument holds for every $\tau$-conic $C$, this shows that $\tilde{\cJ}$ is a universal family on $M^X_0$.
\end{proof}

Consider the period map of GM fourfolds, 
$$\wp_{4}: \mathbf{M}_{4}^{\mathrm{GM}}\longrightarrow\mathscr{D}.$$ 
It is known that the fibers of $\wp_4$ are of dimension four. On the other hand, it is shown in \cite{kuznetsov2019categorical} that GM fourfolds have the equivalent Kuznetsov components if they are in the same fiber of the period map. It is natural to ask if $\Ku(X)$ determines the birational isomorphism class of $X$. In fact, Kuznetsov--Perry propose the following conjecture in \cite[Conjecture 1.9]{kuznetsov2019categorical}.

\begin{conjecture}
\label{K_P_conjecture}
Let $X$ and $X'$ be GM varieties of the same dimension such that there is an equivalence $\Ku(X)\simeq\Ku(X')$, then $X$ is birational to $X'$. 
\end{conjecture}

In \cite[Theorem 1.5]{JLLZ2021gushelmukai}, we prove the conjecture for general GM threefolds. In the current article, we prove this conjecture for very general GM fourfolds.

\begin{theorem}
\label{KP_conjecture}
Let $X$ and $X'$ be very general ordinary GM fourfolds with equivalent Kuznetsov components $\Ku(X)\simeq\Ku(X')$. Then $X$ is either the period partner or the period dual of $X'$. In particular, $X$ is birational to $X'$.  
\end{theorem}

\begin{proof}
Let $\Phi$ denote the equivalence $\Ku(X)\simeq\Ku(X')$, then $\Phi$ induces an isometry between numerical Grothendieck groups. Since $X$ and $X'$ are non-Hodge-special, the isometry will map the canonical rank two lattice  $\langle\Lambda_1,\Lambda_2\rangle $ to $\langle\Lambda_1', \Lambda_2'\rangle$. Thus $\Phi$ induces a bijection~$\phi$ or~$\phi'$ between Bridgeland moduli spaces
\[
\begin{tikzcd}
{\cM_\sigma(\mathcal{K}u(X), \Lambda_1)} \arrow[r, "\phi"] \arrow[rd, "\phi'"] & {\cM_{\sigma'}(\mathcal{K}u(X'),\Lambda_2')}  \\
                                                                  & {\cM_{\sigma'}(\mathcal{K}u(X'), \Lambda_1')}
\end{tikzcd}
\]

By Lemma \ref{uni_fam}, the open  subscheme $M^X_0$ of $\cM_\sigma(\mathcal{K}u(X), \Lambda_1)$ admits a universal family. Then according to the standard argument in \cite[Section 5]{bernardara2012categorical}, $\phi$ or~$\phi'$ is a morphism when restricted to $M^X_0$, hence is actually a birational isomorphism. Since $X$ and $X'$ are very general, by Verbitsky's Torelli theorem \cite{verbitsky}, we obtain that $\phi$ or~$\phi'$ is actually an isomorphism.

In either case, after taking the primitive cohomology on both sides, we have a Hodge isometry $$\langle\Lambda_1,\Lambda_2\rangle^{\perp}\cong\langle\Lambda_1',\Lambda_2'\rangle^{\perp}.$$

Using the result in \cite[Proposition 4.14]{perry2019stability}, this is equivalent to an isomorphism between weight two Hodge structure of $X$ and $X'$
$$H^4(X,\mathbb{Z})_0(1)\cong H^4(X',\mathbb{Z})_0(1).$$
Thus we deduce that $X$ and $X'$ are period partners or period duals (cf.~\cite[Remark 5.28]{debarre2019gushel}). In particular, $X$ is birational to $X'$ by \cite[Corollary 4.16, Theorem 4.20]{debarre2015gushel}.
\end{proof}

\begin{remark}
    Note that we do not assume the equivalence $\Phi\colon \Ku(X)\to \Ku(X')$ to be Fourier--Mukai type, thus we need the existence of the universal family to apply the standard argument in \cite[Section 5]{bernardara2012categorical}.
\end{remark}

\begin{remark}
As a corollary of Theorem~\ref{KP_conjecture}, for very general ordinary GM fourfolds $X$ and $X'$, they are period partners or duals if and only if there exists an equivalence $\Phi:\Ku(X)\simeq \Ku(X')$. In a forthcoming preprint \cite{LPZ22}, the authors prove that any equivalence $\Phi$ between the Kuznetsov components of GM fourfolds is of Fourier--Mukai type. Thus $\Phi$ induces a Hodge isometry between the numerical Grothendieck groups. Using this property, combined with the result in \cite{bayer2022kuznetsov}, 
there is a more general version of Theorem~\ref{KP_conjecture}. More precisely, using \cite[Theorem 5.12]{bayer2022kuznetsov} and the same argument in \cite[Remark 6.16]{PS2022categorical}, one can prove that any two GM fourfolds $X$ and $X'$ are period partners or duals if and only if there exists an equivalence $\Phi:\Ku(X)\simeq\Ku(X')$ such that the induced Hodge isometry $[\Phi]: \cN(\Ku(X))\to \cN(\Ku(X'))$ maps $\langle \Lambda_1, \Lambda_2\rangle$ to~$\langle \Lambda_1', \Lambda_2'\rangle$.

\end{remark}

\bibliography{ref}

\bibliographystyle{alpha}

\end{document}